\documentclass{amsart}
\usepackage{amsmath,amssymb}
\usepackage{color,mathdots}

\newtheorem{theorem}{Theorem}[section]
\newtheorem{lemma}[theorem]{Lemma}
\newtheorem{corollary}[theorem]{Corollary}
\newtheorem{fact}[theorem]{Fact}
\newtheorem{proposition}[theorem]{Proposition}

\theoremstyle{definition}
\newtheorem{example}[theorem]{Example}

\newtheorem{remark}[theorem]{Remark}
\newtheorem{definition}[theorem]{Definition}

\def \d {\delta}
\def \s {\sigma}

\def \nn {\mathfrak n}
\def \Nn {\NN\times \nn}
\def \NNn {\NN\times\NN\times \nn}
\def \NNn {\NN\times \NN\times \nn}
\def \ine {\triangleleft}
\def \ineq {\trianglelefteq}

\def \Ld {{\mathcal L}_\d}
\def \Ldd {{\mathcal L}_{\d,\s}}

\def \Lr {{\mathcal L}_{\d,r}}
\def \Ldd {{\mathcal L}_{\d,\s}}

\def \acl{\operatorname{acl}}
\def \alg{\operatorname{alg}}
\def \tp{\operatorname{\tp}}

\def \ord {\operatorname{ord}}

\def \NN {{\mathbb N_0}}

\def\Ind#1#2{#1\setbox0=\hbox{$#1x$}\kern\wd0\hbox to 0pt{\hss$#1\mid$\hss}
\lower.9\ht0\hbox to 0pt{\hss$#1\smile$\hss}\kern\wd0}

\def\Notind#1#2{#1\setbox0=\hbox{$#1x$}\kern\wd0\hbox to 0pt{\mathchardef
\nn=12854\hss$#1\nn$\kern1.4\wd0\hss}\hbox to
0pt{\hss$#1\mid$\hss}\lower.9\ht0 \hbox to
0pt{\hss$#1\smile$\hss}\kern\wd0}

\title{The theory DCF$_p$A exists for $p>0$}

\author{Kai Ino}
\address{Kai Ino, Department of Mathematics, Meijo University, Tenpaku-ku, Nagoya, Aichi, 468-8502, Japan}
\email{kai.ino@gmath.meijo-u.ac.jp}

\author{Omar Le\'on S\'anchez}
\address{Omar Le\'on S\'anchez, Department of Mathematics, University of Manchester, Oxford Road, Manchester, United Kingdom M13 9PL}
\email{omar.sanchez@manchester.ac.uk}

\date{\today}
\thanks{{\em Acknowledgements}: The second author was partially supported by EPSRC grant EP/V03619X/1}

\begin{document}

\maketitle

\begin{abstract}
We prove that the (elementary) class of differential-difference fields in characteristic $p>0$ admits a model-companion. In the terminology of Chatzidakis-Pillay \cite{CP}, this says that the class of differentially closed fields of characteristic $p$ equipped with a generic differential-automorphism is elementary; i.e., DCF$_p$A exists. Along the way, we provide alternative first-order axiomatisations for DCF (differentially closed fields) and also for DCF$_0$A.
\end{abstract}

\tableofcontents

\section{Introduction}

This paper is concerned with the existence of a model companion of the theory of differential fields in positive characteristic equipped with a distinguished (differential-)automorphism. This fits in the framework of Chatzidakis-Pillay \cite{CP} where one is given a model-complete theory $T$ and asks whether the expanded theory $T_\s$ (specifying that $\s$ is an automorphism) has a model companion $TA$. For us, the situation at hand is when $T=\,$DCF$_p$ for $p>0$. 

\smallskip

It has been over 25 years since \cite{CP} appeared and still there is no general criterion for the existence for the theory $TA$. It is known to exist when $T=\,$ACF \cite{ChatHru,Mac}, also when $T=\,$DCF$_0$ \cite{Bu} and more generally DCF$_{0,m}$ \cite{LS}, and also when $T=\,$SCF$_{e,\lambda}$ \cite{Chatzi}; and there are some existence criterion when $T$ is of finite Morley rank \cite[\S3.11]{CP}. There are also known instances when $TA$ does not exist: when $T$ is ACFA or Psf \cite[\S3.11]{CP}, and also when $T$ has SOP \cite{KS}.

\smallskip

Let us briefly recall what is currently known about companions in theories of fields equipped with a differential and a difference structure. First, recall that a differential field is a field $K$ equipped with an additive map $\d:K\to K$ such that $\d(xy)=\d(x)\,y+x\, \d(y)$, and a difference field is field equipped with a field-endomorphism $\s:K \to K$ (we do not require $\s$ to be surjective). For $p$ zero or prime, let $T_{p,\d,\s}$ denote the theory of fields of characteristic $p$ equipped with a differential and difference structure (note that this theory does {\bf not} require the operators to commute; i.e., $\s$ is not necessarily a differential morphism), and we let $T_{p,\d,\s}^{com}$ be the expanded theory where we impose that $\d$ and $\s$ commute with each other. Here is what we know so far:

\begin{enumerate}
\item [(i)] The theory $T_{0,\d,\s}$ has a model companion DDF$_0$ \cite[Theorem 5.3]{Pie2}.
\smallskip
\item [(ii)] The theory $T_{0,\d,\s}^{com}$ has a model companion DCF$_0$A \cite[Theorem 3.15]{Bu}.
\smallskip
\item [(iii)] For $p>0$, the theory $T_{p,\d,\s}$ does {\bf not} admit a model companion \cite[Theorem 5.2]{Pie2}. In the same paper it is pointed out that this remains true even if we add axioms to $T_{p,\d,\s}$ specifying that $\s$ is an automorphism (not just an endomorphism).
\end{enumerate}

The case that remained open (until now) was whether the theory $T_{p,\d,\s}^{com}$, for $p>0$, admits a model companion. As we point out in Corollary~\ref{finalresult}, this is equivalent to the existence of DCF$_p$A (in the sense described above). The main goal of this paper is to give a proof that such a companion theory exists.
 
 \smallskip

The reader might find a bit misleading that we claim the existence of a model-companion for $T_{p,\d,\s}^{com}$ while $T_{p,\d,\s}$ does not admit one (for $p>0$). However, let us point out that the reasons of noncompaniability of the latter theory are not present in the former. If $(K,\d,\s)$ is an existentially closed model of $T_{p,\d,\s}$, then in \cite[Lemma 5.1]{Pie2} it is shown that $\cap_{i\geq 0}\operatorname{ker}(\d\,\s^i)=K^p$ and $K/\s(K)$ is purely inseparable. Furthermore, in \S5 of the same paper, it is shown that for each $n\in \mathbb N$ there are existentially closed models $(K,\d,\s)$ and $(F,\d',\s')$ of $T_{p,\d,\s}$ such that $\cap_{i<n}\operatorname{ker}(\d\,\s^i)$ is not contained in $K^p$ and $a^{p^n}$ is not in $\s'(F)$ for some $a\in F$; and then either of these two ``type-conditions" can be used to prove the noncompaniability. However, examples such as $K$ and $F$ cannot be existentially closed models of $T_{p,\d,\s}^{com}$. Indeed, in Corollary~\ref{conseq1} below we prove that if $(K,\d,\s)$ is an existentially closed model of $T_{p,\d,\s}^{com}$, then $\operatorname{ker}(\d)=K^p$, and in Lemma~\ref{extendtoauto} that $\s$ must be surjective.

\smallskip

The paper is organised as follows. In Section 2 we present some preliminary results that are relevant for other sections. Then, in Section 3, we prove two finiteness results: one for separable differential extensions and one for separable difference extensions; both of these are crucial for our axioms of DCF$_p$A. In Section~4 we recall the machinery of differential kernels and use it, along with the finiteness result for separable differential extensions, to present alternative axioms for the theory DCF (differentially closed fields with no characteristic specified). We then introduce, in Section 5, the notion of $dd$-kernels; namely, the differential-difference analogue of differential-kernels. We prove that, under certain assumptions, a $dd$-kernel $L_{(r,s)}$ can always be extended to a $dd$-kernel $L_{(\infty,s)}$ and this is used to present alternative axioms for DCF$_0$A. However, we note there that in characteristic $p>0$ these assumptions are not enough to construct a $dd$-kernel extension of the form $L_{\infty,\infty}$ (which is it what we call a $dd$-realisation). Thus, in Section 6, using results of Chatzidakis~\cite{Chatzi} we expand the assumptions (or rather axioms) to guarantee the existence of $dd$-realisations in characteristic $p>0$. Putting all this together, in Section 7, we prove our main result: the theory DCF$_p$A exists.

\smallskip

Given the main result of this paper (existence of DCF$_p$A) and the existence of the theory SCF$_{e,\lambda}$A \cite{Chatzi}, it seems natural to ask whether SDCF$_{\epsilon,\ell}$A exists. Here SDCF$_{\epsilon,\ell}$ denotes the theory of separably differentially closed fields of characteristic $p>0$ of differential-degree of imperfection $\epsilon\in \mathbb N\cup\{\infty\}$ expanded by the differential-lambda functions (this is a model-complete and stable theory, see ~\cite{InoLS}). At the moment we do not know the answer and leave this question for future work.

\smallskip

Another interesting question to address is whether our methods can be extended to the case of several commuting derivations; that is, when equipping the theory DCF$_{p,m}$ of differential fields in characteristic $p>0$ with $m$-many commuting derivations with a differential-automorphism. As we noted above, in characteristic zero, the theory DCF$_{0,m}$A does exist \cite{LS}; however, the methods there are somewhat different to the ones presented here (in \cite{LS} the machinery of differential-characteristic sets was deployed instead of $dd$-kernels). It seems plausible that, in positive characteristic, one could perform a mixture of methods but the issues could be quite subtle and would (potentially) require significant work. As such, we also leave the question of existence of DCF$_{p,m}$A for future work.

\medskip

\noindent {\bf Conventions.} $\mathbb N$ denotes the positive integers while $\NN$ denotes the nonnegative integers.

\medskip

\noindent {\bf Acknowledgements.} We would like to thank Zoe Chatzidakis for useful discussions and for pointing us in the right direction. We would also like to thank the anonymous referee for carefully reading our paper and for their detailed feedback.

\

\section{Some preliminary results}\label{preli}

In this section we set up the notation and terminology that will be used throughout the paper while recalling some basic results and proving some others that will be deployed later on. Unless stated otherwise, fields are of arbitrary characteristic. 

\smallskip

For $(K,\d)$ a differential field (of arbitrary characteristic), we let $C_K$ denote the subfield of $\d$-constants; namely, $C_K=\operatorname{ker}(\d)$. One of the most useful tools to extend derivations to a field extension is the following well known lemma.

\begin{lemma}\label{keyforext}\cite[Theorem 14]{Jacobson64} Let $L/K$ be a field extension and $\d:K\to L$ a derivation.
\begin{enumerate}
\item Let $\bar a, \bar b$ be $n$-tuples from $L$. Then, there is a derivation $\d:K(\bar a)\to L$ extending that on $K$ mapping $\bar a\mapsto \bar b$ if and only if 
$$\sum_{i=1}^n\frac{\partial f}{\partial x_n}(\bar a)\cdot b_i+f^\d(\bar a)=0$$
as $f$ varies in a set of generators of the ideal of vanishing of $\bar a$ over $K$. Here $f^\d$ is the polynomial obtained by applying $\d$ to the coefficients of $f$.
\medskip

\item In particular, if $a\in L$ (i.e., $a$ is a singleton), we have three cases:
\begin{itemize}
\item [(i)] if $a$ is separably algebraic over $K$, then there exists a unique extension to a derivation $\d:K(a)\to L$. This extension maps $a\mapsto \frac{-f^\d(a)}{f'(a)}$ where $f$ is the minimal polynomial of $a$ over $K$.
\item [(ii)] if $a$ is inseparably algebraic over $K$ and the minimal polynomial of $a$ over $K$ has coefficients in $C_K$, then for any $b\in L$ there is an extension to a derivation $\d:K(a)\to L$ such that $\d(a)=b$.
\item [(iii)] if $a$ is transcendental over $K$, then for any $b\in L$ there is an extension to a derivation $\d:K(a)\to L$ such that $\d(a)=b$
\end{itemize} 
\end{enumerate}
\end{lemma}

\

We recall that a differential field $(K,\d)$ of characteristic $p>0$ is called \emph{differentially perfect} if $C_K=K^p$. In \cite[\S II.3]{Kolbook} it is shown that $(K,\d)$ is differentially perfect if and only if every differential field extension is a separable extension. In characteristic zero, every differential field is differentially perfect. We note that every differential field can be embedded in a differentially perfect extension (see Theorem 4 of \cite{Wood73}). However, in positive characteristic, there is \emph{no} canonical way of doing this. In other words, generally there is no differentially perfect closure (i.e., a differentially perfect extension that embeds in every other such). Essentially the reason for this is that a derivation can be extended arbitrarily to $p$-th roots of constants (see Lemma~\ref{keyforext}(ii)).

\medskip

Let $\Ld$ be the language of differential fields. Recall that a differential field is said to be differentially closed if it is existentially closed (in the language $\Ld$) in every differential extension. For a differential field $(K,\d)$, we denote the ring of differential polynomials in variable $x$ over $K$ by $K[x]_\d$. The following fact shows that the class of differentially closed fields is elementary.
 
\begin{fact}\cite{Wood76}
A differential field $(K,\d)$ is differentially closed if and only if $(K,\d)$ is differentially perfect and for every nonzero $f,g\in K[x]_\d$, with $ord(g)<ord(f)$ and $s_f\neq 0$ (here $s_f$ denotes the separant of $f$), there exists $a\in K$ such that $f(a)=0$ and $g(a)\neq 0$.
\end{fact}
 
 The $\Ld$-theory of differentially closed fields, DCF, is the model-companion of the theory of differential fields. It is also the model-completion of the theory of differentially perfect fields. Furthermore, once we specify the characteristic $p$, the theory DCF$_p$ is complete and stable \cite{Wood74,Wood76}. It is also shown there that if $(K,\d)\models$DCF$_p$ and $p>0$, then the underlying field $K$ is a separably closed field of infinite degree of imperfection; in other words, $K\models$SCF$_{p,\infty}$.
 
 In characteristic zero, the theory DCF$_0$ admits quantifier elimination in the language $\Ld$. This is not the case for DCF$_p$ with $p>0$. However, there is a natural language in which DCF$_p$ admits quantifier elimination. One simply needs to add the $p$-th root function on constants. More precisely, for a differentially perfect $(K,\d)$, we let $r:K\to K$ be defined as
$$
r(b)=
\left\{
\begin{matrix}
\; 0 & \text{ if }b\notin C_K \\
\quad b^{1/p} & \text{if } b\in C_K.
\end{matrix}
\right.
$$
After adding this function to the language (and axioms specifying how it is defined), we denote the new language by $\Lr$ and the theory by DCF$_{p}^r$. In \cite{Wood76}, it is shown that DCF$_p^r$ has quantifier elimination. 

It follows from quantifier elimination in the language $\Lr$ that the model-theoretic definable closure and algebraic closure have natural interpretations in DCF$_{p}$. Namely, for $(L,\d)\models$DCF$_p$ and $(K,\d)$ a differential subfield of $L$, the definable closure, $\operatorname{dcl}(K)$, is the differential perfect closure of $K$ in $L$ (i.e., the smallest differential subfield of $L$ containing $K$ which is differentially perfect). The (model-theoretic) algebraic closure, $\operatorname{acl}(K)$, is the separable closure of $\operatorname{dcl}(K)$ in $L$.

\begin{remark}
Let $p>0$. We note that if $(K,\d)$ is a differential subfield of a differentially perfect field $(L,\d)$. Then, the differential perfect closure of $K$ in $L$ can be constructed recursively as follows: let $r:L\to L$ be the $r$-function on $L$, and let $K_0=K$ and $K_{i+1}$ be the differential field generated by $r(C_{K_i})$. Then, $\cup_{i\geq 0}K_i$ is the differential perfect closure of $K$ in $L$.
\end{remark}

By a difference structure on a field $K$ we simply mean an endomorphism $\s:K\to K$ (not necessarily onto). For $A$ a subset of a difference field extension of $K$, we denote the difference field generated by $A$ over $K$ by $K(A)_\sigma$; i.e., this is the field generated over $K$ by $\{\s^i(a):a\in A, i\in \NN\}$. If $K$ comes equipped with a derivation $\d$, we say that $\s$ is a differential endomorphism if $\d\s(a)=\s\d(a)$ for all $a\in K$. By a differential-difference field $(K,\d,\s)$ we mean that $\d$ is derivation and $\s$ is a differential endomorphism; in other words, the differential and difference structure commute on $K$.

\smallskip

The following lemma shows that any differential-difference field can be extended to a differential-difference field where the endomorphism is an automorphism. In particular this shows that an existentially closed differential-difference field has a surjective endomorphism.

\begin{lemma}\label{extendtoauto}
Let $(K,\d,\s)$ be a differential-difference field. Then, there exists a differential-difference extension $(F,\d,\s)$ such that $\s:F\to F$ is surjective.
\end{lemma}
\begin{proof}
Suppose there is $b\in K\setminus \s(K)$. We prove that there is a differential-difference extension $(L,\d,\s)$ such that $b\in \s(L)$. A standard iterative process then yields the desired differential-difference field with $\s$ surjective. We construct a sequence $(c_i)_{i\geq -1}$ (from a field extension of $K$) recursively as follows:

\begin{enumerate}
\item [(i)] Set $c_{-1}=0$.
\item [(ii)] For $n\geq 0$. If $\d^{n}(b)$ is transcendental over $\s(K)(b,\dots,\d^{n-1}(b))$, let $c_{n}$ be a transcendental over $K(c_0,\dots,c_{n-1})$. Otherwise, if $f_{n}$ is the minimal polynomial of $\d^n(b)$ over $\s(K)(b,\dots,\d^{n-1}(b))$, let $c_n$ be a root of $f_{n}^{\s^{-1}}$. The latter polynomial is obtained by replacing $\d^i b$ in the coefficients of $f_n$ with $c_i$ for $i\leq n-1$. 
\end{enumerate}
Let $L=K(c_i:i\geq 0)$. Clearly, there is a field-endomorphism $\s:L\to L$ extending the one on $K$ mapping $c_i\mapsto \d^i(b)$ for all $i\geq 0$; in particular, $\s(c_0)=b$. Since $\s(L)=\s(K)(\d^i b)_{i\geq 0}$ is a differential subfield of $K$ (this uses the fact that $\d$ and $\s$ commute on $K$), we may induce a derivation $\tilde \d$ on $L$ by $\tilde \d=\s^{-1}\d\s$. This derivation extends the one $K$: for $a\in K$ we have $\tilde \d(a)=\s^{-1}\d\s a=\s^{-1}\s\d a=\d a$ (where the previous to last equality uses that $\d$ and $\s$ commute on $K$). Moreover, $\tilde\d $ and $\s$ commute on $L$: for $a\in L$ we have 
$$\s\tilde \d a=\s\s^{-1}\d\s a=\d\s a=\s^{-1}\s\d\s a=\s^{-1}\d\s\s a=\tilde\d\s a$$
where the previous to last equality uses that $\d$ and $\s$ commute on $K$. Thus $(L,\tilde \d,\s)$ is the desired differential-difference field.
\end{proof}

The following lemma is an immediate consequence of the fact that the theories DCF$_0$ and DCF$_p^r$ admit quantifier elimination (recall that the latter theory is in the expanded language $\Lr$). 

\begin{lemma}\label{satext}
Assume $(L,\d)$ is differentially perfect and $(E,\d)$ is a differential extension. Let $\s: L\to E$ be a differential homomorphism. Then, there exists a differential-difference field $(F,\d',\s')$ with $(F,\d')\models$DCF a differential extension of $(E,\d)$ and $\s'|_L=\s$.
\end{lemma}
\begin{proof}
Let $(F,\d)$ be a sufficiently saturated model of DCF extending $(E,\d)$. By quantifier elimination, and using the fact that $(L,\d)$ is assumed to be differentially perfect, the map $\s:L\to E$ is a partial elementary map of $(F,\d)$ and hence, by saturation, it extends to a differential automorphism of $(F,\d)$. 
\end{proof}

In the following example we observe that, in positive characteristic, a differential morphism does not necessarily extend to the differential perfect closure inside a prescribed differential extension.

\begin{example}
Consider the field $\mathbb{F}_p(t,s)$ in two independent indeterminates. Equip it with a differential-difference structure by setting $\d(t)=\d(s)=0$ and $\s(t)=s$ and $\s(s)=t$. Now extend the derivation to $\mathbb{F}_p(t^{1/p}, s^{1/p})$ by settting $\d(t^{1/p})=1$ and $\d(s^{1/p})=0$. Such an extension exists by Lemma~\ref{keyforext}. We now note that $\s$ cannot be extended to a differential endomorphism of $\mathbb{F}_p(t^{1/p}, s^{1/p})$. Indeed, if such an extension were to exist, it would yield
$$1=\s(1)=\s(\d(t^{1/p}))=\d(\s(t^{1/p}))=\d(s^{1/p})=\s(0)=0.$$
\end{example}

\medskip

Nonetheless, we now show that if there is no prescribed differential extension, we can in fact extend the difference structure to \emph{some} differentially perfect extension.

\begin{proposition}\label{diffper}
Let $(L,\d,\s)$ be a differential-difference field. Then, there is a differential-difference extension $(F,\d,\s)$ such that $(F,\d)$ is differentially perfect.
\end{proposition}
\begin{proof}
We may assume that the characteristic is $p>0$. Let $(c_i)_{i\in I}$ be a $p$-basis for $C_L$ over $L^p$. Let $E:=C_L^{1/p}$ inside some algebraic closure of $L$. Letting $x_i\in E$ be the $p$-th root of $c_i$, for $i\in I$, we see that $E=L(x_i:i\in I)$. Clearly, $\s:L\to L$ extends uniquely to a \emph{field}-endomorphism $\s:E\to E$ (note that $E$ has not been equipped with a derivation). Indeed, $y_i:=\s(x_i)$ is the $p$-th root of $\s(c_i)\in C_L$ for all $i\in I$. 

\medskip

Let $\{x_{i,j}:\, i\in I,j\geq 1\}$ be new independent indeterminates over $E$. For ease of notation we set $ x_{i,0}:=x_i$. Extend the derivation $\d$ from $L$ to 
$$F:=L(x_{i,j}:i\in I, j\geq 0)$$ 
by setting $\d(x_{i,j})=x_{i,j+1}$. By construction, $(x_i)_{i\in I}$ is a differential inseparability basis for $F$ over $L$ (in the sense of Kolchin \cite[\S II.9]{Kolbook}). By \cite[Corollary 5, \S II.9]{Kolbook}, the constant field of $F$ is equal to $F^p\cdot C_L$. Since $C_L\leq E^p$, we obtain that $C_F=F^p$. In other words, $(F,\d)$ is differentially perfect. 

\medskip

We now argue that we can extend the differential morphism $\s$ from $L$ to $F$. As we already pointed out above, there is a unique extension to $\s:E\to E$ and we set $y_i:=\s(x_i)$ for all $i\in I$. In order to extend the map to all of $F$, it suffices to argue that $(\d^j y_i:i\in I,j\geq 1)$ are algebraically independent over $E$, as then we may set $\s$ to be the differential morphism mapping $x_{i,j}\mapsto \d^{j}(y_i)$.  

\medskip

Let $n$ and $m$ be positive integers. It suffices to show that 
$$(\d^j y_i: 1\leq i\leq n,1\leq j\leq m)$$ 
are algebraically independent over $E$. Note there is $s\geq n$ such that $y_1,\dots,y_n\in L(x_1,\dots,x_s)$. Furthermore, after possibly re-ordering, we may assume that 
$$L(x_1,\dots,x_s)=L(y_1,\dots,y_n,x_{n+1},\dots,x_s).$$
It follows, from basic properties of derivations, that
$$E(\d^j x_1, \dots,\d^j x_s: 1\leq j\leq m)=E(\d^j y_1,\dots,\d^j y_n,\d^j x_{n+1},\dots,\d^j x_s: 1\leq j\leq m).$$
The transcendence degree of the left-hand-side field over $E$ is $s\cdot m$. As this must also hold for the right-hand-side field, we have that $(\d^j y_i: 1\leq i\leq n,1\leq j\leq m)$ are indeed algebraically independent over $E$.
\end{proof}

\begin{remark}\label{remark11}
We note that a combination of Lemma~\ref{satext} and Proposition~\ref{diffper} says that if $(K,\d,\s)$ is a differential-difference field, then $(K,\d)$ can be extended to $(F,\d)$ a model of DCF such that $\s:K\to K$ is a partial elementary map of $(F,\d)$.
\end{remark}

\begin{corollary}\label{conseq1}
Let $(K,\d,\s)$ be an existentially closed differential-difference field. Then, $(K,\d)$ is differentially perfect. Furthermore, $(K,\d)\models DCF$. In particular, if $char(K)=p>0$, then $K\models $SCF$_{p,\infty}$.
\end{corollary}
\begin{proof}
The fact that $(K,\d)$ is differentially perfect is a consequence of Proposition~\ref{diffper}.

We now show that $(K,\d)$ is existentially closed. Suppose $(L,\d)$ is an extension of $(K,\d)$. By Lemma~\ref{satext}, there exists a differential-difference field extension $(F,\d,\s)$ of $(K,\d,\s)$ such that $(F,\d)$ extends $(L,\d)$. But then, as $(K,\d,\s)$ is existentially closed in $(F,\d,\s)$, $(K,\d)$ is existentially closed in $(F,\d)$, and a posteriori in $(L,\d)$. 

The 'in particular' clause follows from the fact that the underlying field of any model of DCF$_p$ is a model of SCF$_{p,\infty}$.
\end{proof}

Suppose $(K,\d,\s)$ is a differential-difference field and that we are given a differential extension $(L,\d)$ with two differential subfields $E_1$ and $E_2$ extending $K$ such that $\s$ extends to a differential homomorphism $\s:E_1\to E_2$. In characteristic zero, it is a consequence of quantifier elimination of DCF$_0$, that there is a differential-difference field extension $(F,\tilde\d,\tilde\s)$ of $(K,\d,\s)$ such that $(F,\tilde\d)$ extends $(L,\d)$ and $\tilde\s|_{E_1}=\s$. Indeed, as we did in the proof of Lemma~\ref{satext}, take $(F,\d)$ a sufficiently saturated model of DCF$_0$ extending $(L,\d)$, then by quantifier elimination $\s:E_1\to E_2$ is a partial elementary map of $(F,\d)$ and hence can be extended. However, the situation is quite different in positive characteristic as DCF$_p$ does not admit quantifier elimination. We provide an example, in characteristic $p>0$, where an extension $(F,\d,\s)$ as above cannot exist. We note that in the example $(E_1,\d)$ is not differentially perfect nor $\s(E_1)\subseteq E_1$; in either of these cases one can produce the desired extension $(F,\d,\s)$ by applying Lemma~\ref{satext} when $(E_1,\d)$ is differential perfect or Remark~\ref{remark11} when $\s(E_1)\subseteq E_1$.

\begin{example}\label{mainex}
Let $K=\mathbb F_3(t)$ equipped with $\d:=\frac{d}{dt}$ and $\s$ mapping $t\mapsto t+1$. Let $x,y$ be independent indeterminates over $K$. Define a derivation on $K(x,y)$ by $\d(x)=t$ and $\d(y)=t+1$.  One can then check that we have a differential morphism
$$\s: K(x,y^3)\to K(y,x^3), \quad  x\mapsto y, \quad y^3\mapsto x^3.$$
We claim that this cannot be extended to a differential morphism on $y$. Suppose there is such extension. Since $\s(y^3)=x^3$, we get $\s(y)=x$, but then
$$\d(\s(y))=\d(x)=t$$
and
$$\s(\d(y))=\s(t+1)=t+2$$
which is a contradiction.
\end{example}

A reader familiar with the axioms of ACFA or DCF$_0$A might identify that the phenomenon in the above example poses a difficulty with finding the companion DCF$_p$A. More precisely, in (geometric) axiomatisations, say for ACFA, the conclusion of the axioms require points of the form $(a,b)=(a,\s(a))$ in a certain algebraic variety over $(K,\s)$. Inside some big algebraically closed field $F$, this yields a homomorphism $K(a)\to K(b)$ which can be extended to $F$; yielding a difference structure $(F,\s)$. In the current setup, we need to avoid situations as in Example~\ref{mainex}; in other words, we will provide conditions that guarantee that such an extension of $\s$ is possible (see Theorem~\ref{ddrealp}). These conditions are inspired from similar conditions (and difficulties) that appear in the axiomatisation of Chatzidakis \cite{Chatzi} of the theory of separably closed fields with a generic automorphism SCFA$_{e,\lambda}$.

\medskip

We conclude this preliminary section with a couple of examples that exhibit some of the difficulties that one encounters in differential fields of positive characteristic (these phenomena do \emph{not} occur in characteristic zero).

\begin{example}Recall that a proper differential ideal in a differential ring is said to be differentially maximal if it is has no proper differential ideals strictly above it. We give examples of differentially maximal ideals that are not radical and hence have no solutions in any differential field (recall that in characteristic zero all differentially maximal ideals are in fact prime). 
\begin{enumerate}
\item [(i)] In the differential polynomial ring $\mathbb F_2[x]_\d$ the differential ideal generated by $x^2$ and $\d x-1$ is proper. However, it has no solution in any differential field. In fact, it is differentially maximal but it is not radical. In particular, it is not contained in any prime differential ideal. 
\item [(ii)] Let $K=\mathbb F_2(t)$ with trivial derivation. Equip the ring $K[x]$ with the derivation $\d(x)=1$. Then $I=(x^2-t)$ is a proper differential ideal. In fact, it is a prime ideal. Let $L=K(t^{1/2})$ with the trivial derivation. Then, in the base change $K[x]\otimes_K L=L[x]$ with the standard derivation the ideal generated by $I$ is differentially maximal but not radical (since $x^2-t=(x-t^{1/2})^2$). Thus, even if one starts with a prime differential ideal, after base change one may get a differential ideal not contained in any prime differential ideal.
\end{enumerate}
\end{example}

\

\section{Two remarks on separable extensions}\label{remarksep}

In this section we prove two finiteness results: one for separable differential field extensions (Theorem~\ref{finiteness}) and one for separable difference field extensions (Proposition~\ref{fordifference}).

\subsection{On separable differential extensions}

The goal of this section is to prove Theorem~\ref{finiteness}. We work with a differential field $(K,\d)$.

\medskip

Fix a natural number $n\geq 1$. Set $\nn=\{1,\dots,n\}$. Recall that we use $\NN$ to denote the nonnegative integers. For $r\in \NN$, we set
$$\Gamma(r)=\{(\xi,i)\in \Nn: \xi\leq r\}.$$
We sometimes write $\Gamma(\infty)$ for $\Nn$. Given $(\xi,i)$ and $(\tau,j)$ in $\Nn$, we set $(\xi,i)\ineq (\tau,j)$ iff 
$$(\xi, i) \text{ is less than or equal } (\tau,j)$$
in the (left-)lexicographic order. 

Let $(L,\d)$ be a differential extension of $(K,\d)$ and suppose that $L$ is generated (as a differential field over $K$) by an $n$-tuple $\bar a=(a_1,\dots,a_n)$. We write this as $L=K( \bar a)_\d$ and set $a_i^\xi:=\d^\xi a_i$ for all $(\xi,i)\in \Gamma(\infty)$. We say that $a^\tau_j$ is a \emph{leader} of $L$ if $a_j^\tau$ is algebraic over $K(\d^\xi a_i: (\xi,i)\ine (\tau,j))$. Furthermore, if $a_j^\tau$ is separably algebraic over $K(\d^\xi a_i: (\xi,i)\ine (\tau,j))$, we say that $a_j^\tau$ is a \emph{separable leader}; otherwise we say it is an \emph{inseparable leader}. A separable leader $a_j^\tau$ is said to be a \emph{minimal-separable leader} if there is no separable leader $a_j^\xi$ with $\xi<\tau$. 


\begin{theorem}\label{finiteness}
Let $(L,\d)/(K,\d)$ be a differential field extension with $L=K(\bar a)_\d$ and $\bar a\in L^n$. Then, using the above terminology, the set of minimal-separable leaders of $L$ is finite. Furthermore, if the field extension $L/K$ is separable, then the set of inseparable leaders is also finite.
\end{theorem}

\begin{proof}
Note that if $a^\tau_j$ is a separable leader, then $a_j^\eta$ is a separable leader for all $\eta>\tau$. This is because $a_j^\eta=\d^{\eta-\tau}(a_j^\tau)$ and then Lemma~\ref{keyforext}(2)(i) yields that in fact $a_j^\eta\in K(\d^\xi a_i: (\xi,i)\ine (\eta,j))$. It follows that the set of minimal-separable leaders is finite.

\smallskip

Now assume that $L/K$ is separable. Towards a contradiction, assume the set of inseparable leaders is infinite. Set
$$A=\{a^\tau_j: \text{there is a minimal-separable leader $a_j^\xi$ with } \xi \geq \tau\}$$
Note that $A$ is finite. Also, set
$$B=\{a_j^\tau: \text{there is no separable leader }a_j^\xi \text{ with } \xi \geq \tau\}$$
Furthermore, for each $r\in\NN$, we set
$$B(r)=\{a_j^\tau\in B: \tau \leq r\}.$$
We will achieve the desired contradiction after proving two Claims.

\bigskip

\noindent \underline{\bf Claim 1.} There is $r\geq 0$ such that $B(r)$ is algebraically dependent over $K$.
\begin{proof}[Proof of Claim 1]
Let $r_0\geq 0$ such that $B(r_0)$ contains an inseparable leader. Such $r_0$ exists since we are assuming that the set of inseparable leaders is infinite. If $B(r_0)$ is algebraically dependent over $K$, we are done. Otherwise, since $B(r_0)$ contains an inseparable leader, we must have that $B(r_0)$ is algebraically dependent over $K(A)$. By exchange, 
\begin{equation}\label{alge}
\text{there exists }\alpha_0\in A \text{ algebraic over } K(B(r_0)\cup (A\setminus\{\alpha_0\})).
\end{equation}
Now let $r_1\geq 0$ such that $B(r_1)\setminus B(r_0)$ contains an inseparable leader. Again, such $r_1$ exists since we are assuming that the set of inseparable leaders is infinite. If $B(r_1)$ is algebraically dependent over $K$, we are done. Otherwise, we must have that $B(r_1)$ is algebraically dependent over $K(A\setminus \{\alpha_0\})$. Indeed, since $B(r_1)\setminus B(r_0)$ contains an inseparable leader, $B(r_1)\setminus B(r_0)$ is algebraically dependent over $K(B(r_0)\cup A)$; but, by display \eqref{alge} above, it is in fact algebraically dependent over $K(B(r_0)\cup (A\setminus\{\alpha_0\}))$. Since $B(r_0)\subseteq B(r_1)$, we get that $B(r_1)$ is algebraically dependent over $K(A\setminus\{\alpha_0\})$, as claimed.
Now, by exchange again, there exists $\alpha_1\in A\setminus\{\alpha_0\}$ such that $\alpha_1$ is algebraic over the field 
$$K(B(r_1)\cup (A\setminus\{\alpha_0,\alpha_1\})).$$
Continue in this fashion to obtain $r_{i+1}$ and $\alpha_i\in A\setminus \{\alpha_0,\dots,\alpha_{i_1}\}$ such that $B(r_{i+1})$ is algebraically dependent over 
$$K(A\setminus\{\alpha_0,\dots,\alpha_i\}).$$
As $A$ is finite, we eventually run over all its elements, and thus obtain $r$ such that $B(r)$ is algebraically dependent over $K$. This proves the claim.
\end{proof}

Let us briefly recall the notion of separable dependency over the field $K$ (see \cite[\S0.II]{Kolbook} for further details). We say that a family $(d_i)_{i\in I}$ in a field extension is separably dependent over $K$ if there exists a polynomial $f\in K[(t_i)_{i\in I}]$ vanishing at $(d_i)_{i\in I}$ such that there is an $i\in I$ with $\frac{\partial f}{\partial t_i}((d_i)_{i\in I})\neq 0$. Note that this is equivalent to saying that, for some $i_0\in I$, $d_{i_0}$ is separably algebraic over $(d_{i})_{i\neq i_0}$. The family is said to be separably independent over $K$ when there is no such polynomial $f$.

\medskip

Claim 1 and our assumption that $L/K$ is separable yields that $B(r)$ is separably dependent over $K$ (see \cite[\S0.II]{Kolbook}). To conclude the proof we now prove that this is impossible, reaching the desired contradiction.

\bigskip

\noindent \underline{\bf Claim 2.} The set $B$ is separably independent over $K$. 
\begin{proof}[Proof of Claim 2]
After possibly re-ordering the entries of the tuple $\bar a=(a_1,\dots,a_n)$, there is $m\leq n$ so that we may write $B$ as $\{a_i^\xi:(\xi,i)\in \NN\times\mathfrak m\}$ where $\mathfrak m=\{1,\dots,m\}$. Let $\hat a=(a_1,\dots,a_m)$. 
We need to prove that for any $(\tau,j)\in \NN\times\mathfrak m$ and any differential polynomial 
$$f\in K[x_1,\dots,x_m]_\d= K[x_i^\xi: (\xi,i)\in \NN\times \mathfrak m],$$ 
if
\begin{equation}\label{vanish}
f(\hat a)=0
\end{equation}
then
\begin{equation}\label{deri}
\frac{\partial f}{\partial (x_j^\tau)}(\hat a)= 0.
\end{equation}
Note that in the differential polynomial ring $K[x_1,\dots,x_m]_\d$ we are writing $x^\xi_i$ instead of $\d^\xi x_i$. When $\tau>\ord f$, we trivially obtain \eqref{deri}. Thus, we may assume $\tau\leq \ord f$.

We proceed by induction on the number $s=(\ord f- \tau)$. First suppose $s=0$;  namely, $\tau=\ord f$. Towards a contradiction, assume \eqref{deri} does not hold and assume $(\tau,j)$ is maximal such in $(\NN\times \mathfrak m,\ineq)$. Applying $\d$ to both sides of equality \eqref{vanish}, we get
$$0=\d(f(\hat a))=f^\d(\hat a)+\sum_{(\xi,i)\leq (\tau,j)}\frac{\partial f}{\partial x_i^\xi}(\hat a)\cdot \d^{\xi+1}(a_i)$$
This yields that $\d^{\tau+1}a_j=a^{\tau+1}_j$ is a separable leader, but this is impossible as $a_j^\tau\in B$. This establishes the base case $s=0$.

Now assume $s>0$ and the statement holds for any instance when the difference is less than $s$. Let $\mu=\ord f$. From the base case of induction, we see that $\frac{\partial f}{\partial x_i^{\mu}}(\hat a)=0$ for all $i\in \mathfrak m$. Thus, by induction hypothesis for $\frac{\partial f}{\partial x_i^\mu}$, we get
$$\frac{\partial^2 f}{\partial x_j^{\tau+1}\partial x_i^\mu}(\hat a)=0, \quad \text{ for all }i\in \mathfrak m.$$
As a result, setting 
$$h=\frac{\partial f}{\partial x^\mu_1}\cdot x_1^{\mu+1} +\cdots +\frac{\partial f}{\partial x^\mu_m}\cdot x_m^{\mu+1}$$
we see that
$$\frac{\partial h}{\partial x_j^{\tau+1}}(\hat a)=0.$$
Now let 
$$g=f^\d+\sum_{\xi<\mu}\frac{\partial f}{\partial x_i^\xi}\cdot x_i^{\xi+1}.$$
Clearly $\d f=g+h$, and hence, since $\d f(\hat a)=0$ and $h(\hat a)=0$, we get $g(\hat a)=0$. Also note that $\ord g\leq \ord f$, and hence, by inductive hypothesis, we get
$$\frac{\partial g}{\partial x_j^{\tau+1}}(\hat a)=0.$$
Putting all this together, we get
$$\frac{\partial (\d f)}{\partial x_j^{\tau+1}}(\hat a)=\frac{\partial (g+h)}{\partial  x_j^{\tau+1}}(\hat a)=\frac{\partial g}{\partial  x_j^{\tau+1}}(\hat a) +\frac{\partial h}{\partial x_j^{\tau+1}}(\hat a)=0.$$
We now use the following (commutator) formula, which can be easily derived,
$$\left[\frac{\partial}{\partial x_j^{\tau+1}}, \d\right]= \frac{\partial}{\partial x_j^\tau}$$
in other words, applied to $f$ we have
$$\frac{\partial (\d f)}{\partial x^{\tau+1}_j}=\d\left(\frac{\partial f}{\partial x_j^{\tau+1}}\right)+\frac{\partial f}{\partial x^\tau_j}.$$
Plugging $\hat a$, the left-hand-side vanishes. On the other hand, by induction hypothesis we have $\frac{\partial f}{\partial x_j^{\tau+1}}(\hat a)=0$ and hence
$$\d\left(\frac{\partial f}{\partial x_j^{\tau+1}}\right)(\hat a)=0.$$
In conclusion, $\frac{\partial f}{\partial x^\tau_j}(\hat a)=0$, as desired.
\end{proof}
\end{proof}

\begin{remark}
The ``furthermore clause'' of Theorem \ref{finiteness} can be compared to a result of Kolchin's on quasi-separable prime differential ideals of a differential polynomial ring \cite[Lemma 1,\S III.2]{Kolbook}.
\end{remark}

\

\subsection{On separable difference extensions}

The goal of this brief section is to prove Proposition~\ref{fordifference}. We work with a difference field $(K,\s)$. We borrow the notation from the previous (sub)section and adapt the terminology to the difference set up. Namely, for a difference field extension $(L,\s)$ of $(K,\s)$ which is generated (as a difference field over $K$) by an $n$-tuple $\bar a=(a_1,\dots,a_n)$, which we denote by $L=K(\bar a)_\s$ and set $a_i^\xi:=\s^\xi a_i$ for all $(\xi,i)\in \Gamma(\infty)$, we say that $a_j^\tau$ is a leader of $L$ if $a_j^\tau$ is algebraic over $K(a_i^\xi:(\xi,i)\ine(\tau,j))$. Analogous to the previous subsection, we have notions of separable leader, inseparable leader, and minimal-separable leader in this (difference) context. Also, with $L$ as above, for every $r\geq 0$ we let $L_r=K(a_i^\xi: (\xi,i)\in \Gamma(r))$.

\begin{proposition}\label{fordifference}
Let $L=K( \bar a)_\s$ with $\bar a$ an $n$-tuple and $\s(\bar a)\in K(\bar a)^{\text alg}$. If $L/K$ is separable, then all minimal-separable leaders of $L$ are in $L_n$.
\end{proposition}
\begin{proof}
Let $d=$tr.deg$_K L$. Then, as $\s(\bar a)\in K(\bar a)^{\text alg}$, we have $d\leq n$. Thus, it suffices to show that all minimal-separable leaders are in $L_d$. If $d=0$, all elements of $\bar a$ are separably algebraic over $K$ and we are done; so we may assume $d\geq 1$. Let $\tau\geq d$ and fix $1\leq \ell \leq n$. Suppose that $\sigma^\tau a_\ell$ is an inseparable leader. It suffices to show that $\sigma^{\tau+1}a_\ell$ is also an inseparable leader. Recall that, over any field $F$, an element $\alpha$ in a field extension is separably algebraic over $F$ iff $\alpha\in F(\alpha^p)$, see for instance \cite[\S1.6]{Chatzi}. Thus,  showing that $\sigma^{\tau+1}a_\ell$ is an inseparable leader is equivalent to showing that 
\begin{equation}\label{degp}
[K(\s^{\xi}a_i)_{(\xi,i)\ineq (\tau+1,\ell)}: K((\s^{\xi}a_i)_{(\xi,i)\ine (\tau+1,\ell)},\s^{\tau+1} a_\ell^p)]=p
\end{equation}
We now set the stage to prove this below. First, for each $(\eta,k)\ineq (\tau+1,\ell)$ define $E_{(\eta,k)}$ to be the field extension 
$$K((\s^{\xi}a_i)_{(\xi,i) \ineq (\eta,k)}, (\s^{\xi}a_i^p)_{(\eta,k)\ine(\xi,i) \ineq (\tau+1,\ell)})/ K((\s^{\xi}a_i)_{(\xi,i) \ine (\eta,k)}, (\s^{\xi}a_i^p)_{(\eta,k)\ineq(\xi,i) \ineq (\tau+1,\ell)})$$
Note that this extension is of degree 1 or $p$. Also note that the extension in \eqref{degp} is precisely $E_{(\tau+1,\ell)}$ and so we are aiming to show that $\deg E_{(\tau+1,\ell)}=p$. 

By the tower law for field extensions, we have
\begin{equation}\label{tower1}
[K(\s^\xi a_i)_{(\xi,i)\ineq (\tau+1,\ell)}:K(\s^\xi a_i^p)_{(\xi,i)\ineq (\tau+1,\ell)}]=\prod_{(\eta,k)\ineq (\tau+1,\ell)}\deg E_{(\eta,k)}
\end{equation}
We observe that this extension is of degree $p^d$. Indeed, since $K(\s^\xi a_i)_{(\xi,i)\ineq (\tau+1,\ell)}$ is finitely generated and separable over $K$ of transcendence degree $d$, it has a separating transcendence basis over $K$ of size $d$ and hence the degree equals $p^d$.

Similarly, for each  $(\eta,k)\ineq (\tau,\ell)$, we define $F_{(\eta,k)}$ to be the extension
$$K((\s^{\xi}a_i)_{(\xi,i) \ineq (\eta,k)}, (\s^{\xi}a_i^p)_{(\eta,k)\ine(\xi,i) \ineq (\tau,\ell)})/ K((\s^{\xi}a_i)_{(\xi,i) \ine (\eta,k)}, (\s^{\xi}a_i^p)_{(\eta,k)\ineq(\xi,i) \ineq (\tau,\ell)})$$
and note that, since $\sigma^\tau(a_\ell)$ is an inseparable leader, we have that $\deg F_{(\tau,\ell)}=p$; and the tower law yields
\begin{equation}\label{tower2}
[K(\s^\xi a_i)_{(\xi,i)\ineq (\tau,\ell)}:K(\s^\xi a_i^p)_{(\xi,i)\ineq (\tau,\ell)}]=\prod_{(\eta,k)\ineq (\tau,\ell)}\deg F_{(\eta,k)}
\end{equation}
As above, this extension is also of degree $p^d$.

By way of construction, it is straightforward to check that:
$$(\dagger)\quad \text{ if } \deg F_{(\eta,k)}=1, \text{ then }\deg E_{(\eta,k)}=1 \text{ and } \deg E_{(\eta+1,k)}=1.$$

Now, since $\deg F_{(\tau,\ell)}=p$ and the extension in \eqref{tower2} has degree $p^d$, for each $k$ we have that in the sequence
\begin{equation}\label{seq1}
F_{(0,k)}, F_{(1,k)},\dots, F_{(\tau-1,k)}
\end{equation}
there can be at most $d-1$ of them of degree $p$; in other words, there are at least $\tau-d+1$ of them of degree 1. Since $\tau\geq d$, it follows that for each $k$ the sequence in \eqref{seq1} contains at least one element of degree 1. Furthermore, by observation~$(\dagger)$ above, for each $k$ we have that in the sequence
\begin{equation}\label{seq2}
E_{(0,k)}, E_{(1,k)},\dots, E_{(\tau,k)}
\end{equation}
there are at least $\tau-d+2$ of them of degree 1. In other words, for each $k$, the number of extensions of degree 1 in \eqref{seq2} is at least one more than the number of extension of degree 1 in \eqref{seq1}; and the latter has at least one extension of degree 1 as we pointed out above. Running over all $k=1,\dots,n$, we see that the number of extensions of degree 1 in $(E_{(\eta,k)})_{(\eta,k)\ineq(\tau,n)}$ is at least $n$-many more than the number of extensions of degree 1 in $(F_{(\eta,k)})_{(\eta,k)\ineq(\tau-1,n)}$. 

Now, since the extension \eqref{tower2} has degree $p^d$ and $\deg F_{(\tau,\ell)}=p$, in the sequence $(F_{(\eta,k)})_{(\eta,k)\ine(\tau,\ell)}$ there are $d-1$ of them of degree $p$; in other words, there are $$(n\,\tau+\ell-1 )-(d-1)=n\,\tau +\ell-d$$ 
of them of degree 1. By the concluding observation in the previous paragraph, in the sequence $(E_{(\eta,k)})_{(\eta,k)\ine(\tau+1,\ell)}$ there are at least $n\, \tau +\ell - d+n$ of them of degree~1. In other words, there are at most 
$$(n\,(\tau+1)+\ell -1)-(n\,\tau +\ell-d+n)= d-1$$
of them of degree $p$ (in fact there are precisely $d-1$). However, the extension \eqref{tower1} has degree $p^d$, and so we must have that $E_{(\tau+1,\ell)}$ is of degree $p$, as desired. 
\end{proof}

We now provide two examples. The first one shows that Proposition~\ref{fordifference} fails without the assumption of separability (of $L/K$). The second shows that the number $n$ in $L_n$ is optimal, in the sense that in general the conclusion of the proposition does not hold for $L_r$ with $r<n$.

\begin{example}
Let $m\geq1$. Let $K=\mathbb F_p(t_1,\dots,t_m)$ with endomorphism $\s(t_i)=t_{i+1}$ for $i<m$ and $\s(t_m)=t_1^p$. Now let $L=K(t_1^{1/p},\dots,t_m^{1/p})$ equipped with the unique morphism extending that on $K$; i.e., $\s(t_i^{1/p})=t_{i+1}^{1/p}$ for $i<m$ and $\s(t_m^{1/p})=t_1$. Note that $L$ is generated, as a difference field over $K$, by the singleton  $t_1^{1/p}$; that is, $L=K(t_1^{1/p})_\s$. However, there is a minimal-separable leader in $L_m\setminus L_{m-1}$. Indeed, for $i<m$, we have $\s^i(t_1^{1/p})=t_{i+1}^{1/p}$ and hence $\s^i(t_1^{1/p})$ is an inseparable leader. On the other hand, $\s^{m}(t_1^{1/p})=\s(t_m)^{1/p}=t_1$, and so $\s^{m}(t_1^{1/p})$ is a minimal-separable leader. 
\end{example}

\begin{example}
Let $K=\mathbb F_p$ and $L=\mathbb F_p(t_1,\dots,t_n, t_1^{p^{n-1}}, t_2^{1/p},\dots,t_n^{1/p})$ with the unique endomorphism determined by $\s(t_i)=t_{i+1}^{1/p}$ for $i<n$ and $\s(t_n)=t_1^{p^{n-1}}$. Note that $L/K$ is separable, $L=K(t_1,\dots,t_n)_\s$, and $\s(t_1,\dots,t_n)\in K(t_1,\dots,t_n)^{\alg}$; that is, all assumptions of Proposition~\ref{fordifference} are satisfied. Hence all minimal-separable leaders of $L$ are in $L_n$, and we now argue that there is in fact one in $L_{n}\setminus L_{n-1}$. Indeed, for $i<n$, $\s^i(t_1)=t_{i+1}^{1/p^i}$ and so this is an inseparable leader. On the other hand, $\s^n(t_1)=\s(t_{n}^{1/p^{n-1}})=t_1$, and so this is a minimal-separable leader.
\end{example}

\

\section{Differential kernels}

The notion of differential kernel (in characteristic zero) was introduced by Lando~\cite{Lando}. In this section we recall the notions that will be relevant to us (our presentation can be compared to that of Pierce in \cite{Pie} where the terminology ``extensions meeting the differential condition'' was used).

We fix a differential field $(K,\d)$, of arbitrary characteristic, and also a natural number $n\geq 1$. As in the previous section we set $\nn=\{1,\dots,n\}$ and use $\NN$ to denote the nonnegative integers. Recall that for $r\in \NN$, we set
$$\Gamma(r)=\{(\xi,i)\in \Nn: \xi\leq r\}.$$
We will allow $r=\infty$ in which case $\Gamma(\infty)=\Nn$. Also recall that, given $(\xi,i)$ and $(\tau,j)$ in $\Nn$, we set $(\xi,i)\ineq (\tau,j)$ iff $(\xi, i)$ is less than or equal $(\tau,j)$ in the (left-)lexicographic order. 

In this section we consider field extensions of the form
\begin{equation}\label{extkernel}
L_r=K(a_i^{\xi}: (\xi,i)\in \Gamma(r)), \quad \text{ for some $r\geq 0$.}
\end{equation}
Occasionally, we will consider extensions of the form 
$$L_{(\tau,j)}=K(a_i^\xi: (\xi,i)\ine (\tau,j))$$ 
for some $(\tau,j)\in \Nn$. Here the $a_i^\xi$'s are just indexing the generators of $L$ as a field extension of $K$; that is, when we reference $a_i^\xi$ we are also referring to $(\xi,i)$. 

For $L_r$ as in \eqref{extkernel}, for any $0\leq s\leq r$, we will write $L_s$ for the subfield 
$$L_s=K(a_i^{\xi}: (\xi,i)\in \Gamma(s)).$$

We say that $a_j^\tau$ a \emph{leader} of $L_r$ if $a_j^\tau$ is algebraic over $K(a_i^\xi: (\xi,i)\ine (\tau,j))$. Furthermore, if $a_j^\tau$ is separably algebraic over $K(a_i^\xi: (\xi,i)\ine (\tau,j))$, we say that $a_j^\tau$ is a \emph{separable leader}; otherwise we say it is an \emph{inseparable leader}. A separable leader $a_j^\tau$ is said to be a \emph{minimal-separable leader} if there is no separable leader $a_j^\xi$ with $\xi<\tau$. Note that all these notions make sense when $r=\infty$ and they coincide with the notions of leaders from \S\ref{remarksep}.

\begin{definition}\label{defdiffker} \
\begin{enumerate}
\item Let $r\geq 1$. A field extension of the form
$$L_r=K(a_i^{\xi}: (\xi,i)\in \Gamma(r))$$
is said to be a differential-kernel (of length $r$) in $n$-variables over $(K,\d)$ if there exists a derivation $\d:L_{r-1}\to L_r$ extending $\d$ on $K$ such that $\d(a_i^\xi)=a_i^{\xi+1}$ for all $\xi\in \Gamma(r-1)$. 

\item In case the extension has the form $$L_{(\tau,j)}=K(a_i^\xi: (\xi,i)\ine (\tau,j))$$ for some $(\tau,j)\in \Nn$, we say that $L_{(\tau,j)}$ is a differential-kernel if there exist a derivation $\d:K(a_i^\xi: (\xi+1,i)\ine (\tau,j))\to L_{(\tau,j)}$ extending $\d$ on $K$ such that $\d(a_i^\xi)=a_i^{\xi+1}$  when $(\xi+1,i)\ine (\tau,j)$. 

\item Unless otherwise stated (or if it is clear from the context) differential-kernels will have the form of part (1). 


\item We say that a differential-kernel $L_r$ has a differential-realisation if there exists a differential field extension $(E,\d')$ of $(K,\d)$ such that $L_r$ is a subfield of $E$ and $\d'(a_i^{\xi})=a_i^{\xi+1}$ for all $(\xi,i)\in \Gamma(r-1)$. Another way to say this is that $L_r$ can be extended to a differential-kernel $L_\infty$ (of length $\infty$); note that the latter is in fact a differential field generated by $(a_1^0,\dots,a_n^0)$.
\end{enumerate}
\end{definition}

\smallskip

\begin{proposition}\label{diffreal}
Let $L_r$ be a differential-kernel over $(K,\d)$ in $n$-variables. Suppose that all inseparable leaders are in $L_{r-1}$. Then, $L_r$ has a differential-realisation $(E,\d)$ of the form 
$E=K(a_i^\xi: (\xi,i)\in \Gamma(\infty))$ with $a_i^\xi:=\d^\xi a_i$
 and such that the minimal-separable leaders and inseparable leaders of $E$ coincide with those of $L_r$.
\end{proposition}
\begin{proof}
Let $L_r=K(a_i^{\xi}: (\xi,i)\in \Gamma(r))$. Let $(\tau,l)\in \Nn$ with $\tau> r$ and assume we have built a differential-kernel $L_{(\tau,l)}=K(a_i^\xi:(\xi,i)\ine (\tau,l))$ extending $L_r$ with the property that all minimal separable leaders and all inseparable leaders of $L_{(\tau,l)}$ coincide with those of $L_r$. This is of course true when $(\tau,l)=(r+1,1)$ as $L_{(r+1,1)}=L_r$. We now wish to specify a value for $a_l^\tau$ that yields a differential-kernel structure on $K(a_i^\xi:(\xi,i)\ineq (\tau,l))$ such that the latter has the same minimal separable leaders and inseparable leaders as $L_{(\tau,l)}$. We consider two cases:

\

\noindent {\bf Case 1.} Assume $a_{l}^{\tau-1}$ is not a leader. Namely, $a_{l}^{\tau-1}$ is transcendental over $K(a_i^\xi:(\xi,i)\ine (\tau-1,l))$. In this case, we set $a_{l}^\tau$ to be a new generic (i.e., transcendental) element. This yields a well defined extension of the derivation to $K(a_i^\xi:(\xi,i)\ine (\tau,l))$ mapping $a_l^{\tau-1}\mapsto a_l^{\tau}$ (see Lemma~\ref{keyforext}(iii)). This gives us a differential-kernel structure that does not add any leaders. 

\

\noindent {\bf Case 2.} Assume $a_{l}^{\tau-1}$ is a leader. By assumption all inseparable leaders are in $L_{r-1}$, and so, since $\tau-1\geq r$, we have that $a_{l}^{\tau-1}$ is a separable leader. Namely, it is separably algebraic over $K(a_i^\xi:(\xi,i)\ine (\tau-1,l))$. Let $f$ be its minimal polynomial. Then, we set 
$$a_l^\tau=-\frac{f^\d(a_l^{\tau-1})}{f'(a_l^{\tau-1})}\in K(a_i^\xi:(\xi,i)\ine (\tau,l)).$$
This yields a well defined extension of the derivation to $K(a_i^\xi:(\xi,i)\ine (\tau,l))$ mapping $a_l^{\tau-1}\mapsto a_l^{\tau}$ (see Lemma~\ref{keyforext}(i)). This gives us a differential-kernel structure and note that we have added a separable leader (namely $a_l^\tau$) which is \emph{not} a minimal separable leader. 

\

Recursively, this procedure yields the desired differential-realisation of $L_r$.
\end{proof}

We note that the above proposition, together with Theorem~\ref{finiteness}, yields an alternative axiomatisation of DCF. The proof also serves as a warm-up for the proof of companiability of differential-difference fields in Theorem~\ref{thecomp} below.

\begin{corollary}
Let $(K,\d)$ be a differential field. Then, $(K,\d)\models DCF$ if and only if the following two conditions are satisfied:
\begin{enumerate}
\item $(K,\d)$ is differentially perfect.

\item For each $n\geq 1$ and $r\geq 1$, if $L_r=K(a_i^{\xi}: (\xi,i)\in \Gamma(r))$ 
is a differential-kernel over $(K,\d)$ in $n$-variables such that all inseparable leaders are in $L_{r-1}$, then $(a_i^\xi:(\xi,i)\in \Gamma(r))$ has an algebraic specialisation over $K$ of the form $(\d^\xi c_i:(\xi,i)\in \Gamma(r))$ for some $n$-tuple $(c_1,\dots,c_n)$ from $K$. 
\end{enumerate}
\end{corollary}
\begin{proof}
(1)$\Rightarrow$(2) Since $(K,\d)\models DCF$, it is differentially perfect. Now suppose that $L_r$ is a differential-kernel over $(K,\d)$ in $n$-variables satisfying condition (2). By Theorem~\ref{diffreal}, $L_r$ has a differential-realisation of the form $E=K(a_i^\xi: (\xi,i)\in \Gamma(\infty))$. Since $(K,\d)$ is existentially closed in $(E,\d)$, the Zariski-locus of 
$$(a_i^\xi:(\xi,i)\in \Gamma(r))=(\d^\xi a_i:(\xi,i)\in \Gamma(r))$$
over $K$ has a point of the form $(\d^\xi c_i:(\xi,i)\in \Gamma(r))$ for some $n$-tuple $(c_1,\dots,c_n)$ from $K$. This yields the desired specialisation.

\

(2)$\Rightarrow$(1) Let $\phi(\bar x)$ be a quantifier-free $\Ld(K)$-formula with a realisation $\bar a$ in some extension $(F,\d)$. We may assume that $\phi(\bar x)$ is of the form 
$$\psi(\d^\xi x_i: \xi\leq s, i\leq n)$$ 
for some $n,s\geq 1$ and $\psi$ a system of polynomial equations over $K$.

Consider the differential field
$$L_\infty=K(\d^\xi a_i: (\xi,i)\in \Nn).$$
Since $(K,\d)$ is differentially perfect, $L/K$ must be a separable field extension; and hence, by the ``furthermore clause'' of Theorem~\ref{finiteness}, the set of inseparable leaders is finite. Thus, we may choose $r\geq s$ large enough such that: all inseparable leaders of $L_\infty$ are in $L_{r-1}$. 

Then, $L_r$ is a differential-kernel satisfying condition (2), and thus there is an algebraic specialisation of $(a_i^\xi:(\xi,i)\in \Gamma(r))$ of the form $(\d^\xi c_i:(\xi,i)\in \Gamma(r))$ for some $n$-tuple $\bar c=(c_1,\dots,c_n)$ from $K$. Then, the tuple
$$(\d^\xi c_i:(\xi,i)\in \Gamma(s))$$
is a solution of $\psi$ and consequently $\bar c$ is realisation of $\phi$ in $K$.
\end{proof}

\begin{remark}
Condition (2) in the above corollary is indeed elementary; namely, it can be expressed as first-order axioms in the language of differential fields. In the proof of Corollary~\ref{finalresult} below we provide a few references of this (nowadays) standard procedure.
\end{remark}

\

\section{$dd$-kernels and $dd$-realisations}\label{onddkernels}

In this section we introduce the notions of \emph{dd-kernels} and \emph{dd-realisations} (here $dd$ stands for differential-difference). We do this as a mixture of the notions of differential kernels~\cite{Lando} and difference kernels~\cite[\S5.2]{Levin}.

\smallskip

We fix a differential-difference field $(K,\d,\s)$, of arbitrary characteristic, and also a natural number $n\geq 1$. As before, set $\nn=\{1,\dots,n\}$. Recall that we use $\NN$ to denote the nonnegative integers. For $r$, $s\in \NN$, we set
$$\Gamma(r,s)=\{(\xi,u,i)\in \NN\times \NN \times \nn: \xi\leq r\text{ and }u\leq s\}.$$
We allow $r$ or $s$ to take the value $\infty$; for instance, we write $\Gamma(\infty,\infty)$ for $\NNn$. Also, given $(\xi,u,i)$ and $(\tau,v,j)$ in $\NNn$, we set $(\xi,u,i)\ineq (\tau,v,j)$ iff $(\xi,u, i)$ is less than or equal $(\tau,v,j)$ in the (left-)lexicographic order. 


We will consider field extensions of the form
\begin{equation}\label{ddextkernel}
L_{(r,s)}=K(a_i^{\xi,u}: (\xi,u,i)\in \Gamma(r,s)), \quad \text{ for some $r,s\geq 0$.}
\end{equation}
Occasionally, we will consider extensions of the form 
\begin{equation}\label{ddext2}
L_{(\tau,v,j)}=K(a_i^{\xi,u}: (\xi,u,i)\ine (\tau,v,j))
\end{equation}
for some $(\tau,v,j)\in \Nn$. Here the $a_i^{\xi,u}$'s are indexing the generators of $L_{(r,s)}$ as a field extension of $K$. For any $0\leq r'\leq r$ and $0\leq s'\leq s$, we will write $L_{r',s'}$ for the subfield 
$$L_{(r',s')}=K(a_i^{\xi,u}: (\xi,u,i)\in \Gamma(r',s')).$$

\medskip

We say that $a^{\tau,v}_j$, with $(\tau,v,j)\in \Gamma(r,s)$, is a \emph{leader} of $L_{r,s}$ if $a_j^{\tau,v}$ is algebraic over $K(a_i^{\xi,u}: (\xi,u,i)\ine (\tau,v,j))$. Furthermore, if $a_j^{\tau,v}$ is separably algebraic over $K(a_i^{\xi,u}: (\xi,u,i)\ine (\tau,v,j))$, we say that $a^{\tau,v}_j$ is a \emph{separable leader}; otherwise we say it is an \emph{inseparable leader}. In case $a^{\tau,v}_j$ is a separable leader, we say that it is a \emph{minimal-separable leader} if there is no separable leader $a^{\xi,u}_j$ with $(\xi,u)<(\tau,v)$ in the product order of $\NN\times \NN$. Note that all these notions make sense when either $r=\infty$ or $s=\infty$.



\begin{definition} \
\begin{enumerate}
\item Let $r,s\geq 1$. A field extension of the form
$$L_{(r,s)}=K(a_i^{\xi,u}: (\xi,u,i)\in \Gamma(r,s))$$
is said to be a $dd$-kernel of length $(r,s)$ in $n$-variables over $(K,\d,\s)$ if there exists a derivation $\d:L_{r-1,s}\to L_{r,s}$ and a field-homomorphism $\s:L_{r,s-1}\to L_{r,s}$ extending $\d$ and $\s$ on $K$, respectively, such that $\d(a_i^{\xi,u})=a_i^{\xi+1,u}$, for all $(\xi,u)\in \Gamma(r-1,s)$, and $\s(a_i^{\xi,u})=a_i^{\xi,u+1}$ for all $(\xi,u)\in \Gamma(r,s-1)$. 

\item In case the extension has the form 
$$L_{(\tau,u,j)}=K(a_i^{\xi,u}: (\xi,u,i)\ine (\tau,v,j))$$
for some $(\tau,v,j)\in \Nn$, we say that $L_{(\tau,u,j)}$ is a $dd$-kernel if there exist a derivation $\d:K(a_i^{\xi,u}: (\xi+1,u,i)\ine (\tau,j))\to L_{(\tau,j)}$ and a field-homomorphism $\s:K(a_i^{\xi,u}: (\xi,u+1,i)\ine (\tau,j))\to L_{(\tau,j)}$ extending $\d$ and $\s$ on $K$, respectively, such that $\d(a_i^{\xi,u})=a_i^{\xi+1,u}$, when $(\xi+1,u,i)\ine (\tau,v,j)$, and $\s(a_i^{\xi,u})=a_i^{\xi,u+1}$ when $(\xi,u+1,i)\ine (\tau,v,j)$. 

\item Unless otherwise stated (or if it is clear from the context) $dd$-kernels will have the form of part (1). 

\item We say that a $dd$-kernel $L_{(r,s)}$ has a $dd$-realisation if there exists a differential-difference field extension $(E,\d',\s')$ of $(K,\d,\s)$ with $L_{(r,s)}$ a subfield of $E$ such that $\d'(a_i^{\xi,u})=a_i^{\xi+1,u}$, for all $(\xi,u,i)\in \Gamma(r-1,s)$, and $\s'(a_i^{\xi,u})=a_i^{\xi,u+1}$ for all $(\xi,u,i)\in \Gamma(r,s-1)$.  Equivalently, $L_{(r,s)}$ can be extended to a $dd$-kernel $L_{(\infty,\infty)}$; note that the latter is in fact a differential-difference field generated by $(a_1^{0,0},\dots,a_n^{0,0})$.
\end{enumerate}
\end{definition}

Clearly any $dd$-kernel $L_{(r,s)}$ can be thought of as a differential-kernel (in the sense of Definition~\ref{defdiffker}) of length $r$ in $n(s+1)$ many variables (simply forgetting $\s$).

\medskip

Note that if $L_{(r,s)}$ is a $dd$-kernel, then 
$$L_{(r,s)}^a:=L_{(r,s-1)}$$
and 
$$L_{(r,s)}^b:=K(a_i^{\xi,u}: (\xi,u,i)\in \Gamma(r,s) \text{ and } u\geq 1)$$
are both $dd$-kernels of length $(r,s-1)$ with $\s(L_{(r,s)}^a)=L_{(r,s)}^b$ where $\s$ is the field-homomorphism of $L_{(r,s)}$.
Similarly, for $L_{(\tau,v,j)}$ as in \eqref{ddext2},
$$L_{(\tau,v,j)}^a:= K(a_i^{\xi,u}: (\xi,u,i)\ine (\tau,v,j) \text{ and } u\leq s-1)$$
and
$$L_{(\tau,v,j)}^b:= K(a_i^{\xi,u}: (\xi,u,i)\ine (\tau,v,j) \text{ and } u\geq 1).$$
are also $dd$-kernels.

\begin{definition}
Given a $dd$-kernel $L_{(r,s)}$, we say that $a^{\tau,u}_j$ is an $a$-leader, separable $a$-leader, inseparable $a$-leader, minimal-separable $a$-leader if $a^{\tau,u}_i$ is such in the $dd$-kernel $L_{(r,s)}^a$. Similarly for the notion of $b$-leader, etc.
\end{definition}

We now prove a result that gives sufficient conditions that guarantee that a $dd$-kernel $L_{(r,s)}$ can be extended to a $dd$-kernel $L_{(\infty,s)}$. The reader might wish to compare this with differential-kernel version in Theorem~\ref{diffreal}.

\begin{theorem}\label{ddreal}
Let $n,s \geq 1$, $M\in \NN$, and $r\geq(ns+1)(M+1)$. Suppose $L_{(r,s)}$ is a $dd$-kernel in $n$-variables such that all minimal separable leaders, inseparable leaders, minimal separable $a$-leaders, and inseparable $a$-leaders are in $L_{(M,s)}$. Then, $L_{(r,s)}$ can be extended to a $dd$-kernel $L_{(\infty,s)}$.
\end{theorem}

\begin{proof}
It suffices to prove the case $s=1$ as, by the usual linearisation process, every $dd$-kernel $L_{(r,s)}$ can be transformed into a $dd$-kernel of length $(r,1)$, say $L'_{(r,1)}$, by adding more variables. In particular, a $dd$-extension of $L'_{(r,1)}$ of length $(\infty,1)$ will yield such a extension for $L_{(r,s)}$. Thus, we assume $s=1$. 

To simplify notation we will write $a_i^\xi$ for $a_i^{\xi,0}$ and $b_i^\xi$ for $a_i^{\xi,1}$, and we will omit $s$ from the notation. So we may write our $dd$-kernel as $L_{r}=K(a_i^\xi,b_i^{\xi}:(\xi,i)\in \Gamma(r))$. Recall that we may view $L_r$ as a differential-kernel in $2n$ many variables.

\medskip

As all inseparable leaders are in $L_M$ and $M<r$, by Theorem \ref{diffreal}, there is a differential-realisation $(E,\d)$ of $L_r$ having the same minimal separable leaders and inseparable leaders as $L_r$. Namely, all are in $L_M=E_M$. We may write $E=K(a_i^{\xi},b_i^\xi: (\xi,i)\in \Nn)$ where $\d a_i^\xi=a_i^{\xi +1}$ and $\d b_i^\xi=b_i^{\xi+1}$.

\medskip

We will prove that we can extend $\s$ to equip $E$ with a $dd$-kernel structure extending that on $L_r$. To do so, we first prove two claims.

\bigskip

\noindent \underline{\bf Claim 1.} All minimal-separable and inseparable $a$-leaders of $E$ are in $E_M$. 
\begin{proof}[Proof of Claim 1]
By assumption, all minimal-separable and inseparable $a$-leaders of $E_r$ are in $E_M$. Now let $(\tau,l)\in \Nn$ with $\tau> r$ and assume that we have shown that all minimal-separable and inseparable $a$-leaders of
$$E_{(\tau,l)}=K(a_i^\xi,b_i^\xi:(\xi,i)\ine (\tau,l))$$
are in $E_M$. It suffices to show that if $a^{\tau-1}_l$ is not an $a$-leader, then $a^\tau_l$ is not an $a$-leader. By construction of $E$, this holds when $a^{\tau-1}_l$ is not a leader. Thus, we may assume that $a^{\tau-1}_l$ is a leader (which must be a separable leader, again by construction of $E$). Since all minimal-separable leaders of $E$ are in $E_M$, we must have that $a^{M+1}_l$ is a separable leader. 

For each $m\in \NN$, set 
$$B_m=\{(\eta,j)\in \Gamma(m): \eta>M \text{ and there is no leader $(\xi,j)$ of $E$ with $\xi\geq \eta$}\}.$$ 
Then, for each $\mu$ with $M+1\leq \mu\leq r$, we have that
$$a^{\mu}_l\in K(a^\xi_i,b^{\eta}_j: (\xi,i)\ine (\mu,l) \text{ with }\eta\leq M\text{ or }(\eta,j)\in B_\mu).$$
Suppose towards a contradiction that for each such $\mu$ we have
$$a^{\mu}_l\in K(a^\xi_i,b^{\eta}_j: (\xi,i)\ine (\mu,l) \text{ and }\eta\leq M).$$
This would imply that the transcedence degree of $a^{M+1}_l,\dots,a^{r}_l$ over 
$$K(a^\xi_i:  (\xi,i)\ine (r,l) \text{ with }(\xi,i)\notin \{(M+1,l),\dots,(r-1,l)\})$$
is at most $n(M+1)$. However, as all minimal-separable and inseparable $a$-leaders of $E_r$ are in $E_M$, we must have that this transcendence degree must be at least 
$$r-M\geq (n+1)(M+1) -M=n(M+1)+1.$$
This gives us the desired contradiction (note that here we used the assumption that $r\geq(n+1)(M+1)$). Hence, there must be $\mu$ with $M+1\leq \mu\leq r$ such that $a^{\mu}_l$ is in
$$K(a^\xi_i,b^{\eta}_j: (\xi,i)\ine (\mu,l) \text{ with }\eta\leq M\text{ or }(\eta,j)\in B_\mu).$$
but not in
$$K(a^\xi_i,b^{\eta}_j: (\xi,i)\ine (\mu,l) \text{ and }\eta\leq M).$$
Choose $\mu$ minimal such. This means that
$$a^{\mu-1}_l\in K(a^\xi_i,b^{\eta}_j: (\xi,i)\ine (\mu-1,l) \text{, and }\eta\leq M).$$
Since $a^{\mu}_l=\d(a^{\mu-1}_l)$, it follows, using basic rules of derivations, that there is $(\chi,k)\in B_\mu$ such that
$$a^{\mu}_l=\alpha\,  b^\chi_k+\beta_0$$
for some $\alpha$ and $\beta_0$ in
$$K(a^\xi_i,b^{\eta}_j: (\xi,i)\ine (\mu,l) \text{ and }\eta\leq M\text{ or }(\eta,j)\in B_\mu\text{ with }(\eta,j)\ine (\chi,k)).$$
with $\alpha\neq 0$. By applying the derivation $s=\tau-\mu$ many times to $a^\mu_l$ we get
\begin{equation}\label{express}
a^\tau_l=\alpha\, b^{\chi+s}_k+\beta_1
\end{equation}
with $\beta_1$ in 
$$K(a^\xi_i,b^{\eta}_j: (\xi,i)\ine (\tau,l) \text{ and }\eta\leq M\text{ or }(\eta,j)\in B_\tau\text{ with }(\eta,j)\ine (\chi+s,k)).$$
Since $(\chi,k)\in B_\mu$, we have that $b^{\chi+s}_k$ is not a leader of $E$; and hence, by equality~\eqref{express}, we see that $a^\tau_l$ cannot be an $a$-leader. This proves Claim 1.
\end{proof}

The second claim states the analogous property for $b$-leaders.

\bigskip

\noindent \underline{\bf Claim 2.} All minimal-separable and inseparable $b$-leaders of $E$ are in $E_M$. 
\begin{proof}[Proof of Claim 2]
The proof is exactly (or symmetric rather) as that of Claim 1 once we note that all minimal-separable and inseparable $b$-leaders of $E_r$ are in $E_M$. This is indeed the case as, by definition of $dd$-kernel, there is $\sigma:L_{r}^a\to L_{r}^b$ a field-homomorphism such that $\s(a_i^\xi)=b^\xi_i$ (note that $L_{r}^a=E_{r}^a$ and $L_{r}^b=E_{r}^b$).
\end{proof}

We now prove that we can extend $\s$ to equip $E$ with a $dd$-kernel structure extending that on $L_r$. Let $(\tau,l)\in \Nn$ with $\tau> r$ and assume we have a $dd$-kernel structure on 
$$E_{(\tau,l)}=K(a_i^\xi,b_i^\xi:(\xi,i)\ine (\tau,l))$$ 
extending that in $L_r$. We now need to show that there is a $dd$-kernel structure on $K(a_i^\xi,b_i^\xi:(\xi,i)\ineq (\tau,l))$ extending that of $E_{(\tau,l)}$. We consider cases:

\medskip

\noindent {\bf Case 1.} Assume $a_{l}^{\tau-1}$ is an $a$-leader. By Claim 1, all inseparable $a$-leaders are in $E_M$, and so, since $\tau-1\geq r$, we have that $a_l^{\tau-1}$ is a separable $a$-leader. Namely, it is separably algebraic over $E_{(\tau-1,l)}^a$. Letting $f(t)$ be its minimal polynomial, we get that
$$a_l^\tau=-\frac{f^\d(a_l^{\tau-1})}{\frac{df}{dt}(a_l^{\tau-1})}.$$
On the other hand, since $\s:E_{(\tau,l)}^a\to E_{(\tau,l)}^b$ is a field-isomorphism, we have that $b_l^{\tau-1}$ is separably algebraic over $E_{(\tau-1,l)}^b$ with minimal polynomial $f^\s(t)$, and hence
$$b_l^\tau=-\frac{(f^\s)^\d(b_l^{\tau-1})}{\frac{d(f^\s)}{dt}(b_l^{\tau-1})}=-\frac{(f^\d)^\s(b_l^{\tau-1})}{(\frac{df}{dt})^\s(b_l^{\tau-1})}.$$
It follows that $\s(a_l^\tau)=b_l^\tau$. Thus, in this case, there is no need to further extend the isomorphism, the one on $E_{(\tau,l)}^a$ is already defined on $a_l^\tau$ and maps it to $b_l^\tau$, as desired.

\medskip

\noindent {\bf Case 2.} Assume $a_{l}^{\tau-1}$ is not an $a$-leader. Since $\s:E_{(\tau,l)}^a\to E_{(\tau,l)}^b$ is a field-isomorphism, it follows that $b^{\tau-1}_l$ is not a $b$-leader. Thus, by Claim 1 and Claim 2, respectively, we have that $a^\tau_l$ is not an $a$-leader and $b^\tau_l$ is not a $b$-leader. In other words, $a^\tau_l$ is transcendental over $L_{(\tau,l)}^a$ and $b^\tau_l$ is transcendental over $L_{(\tau,l)}^b$. This implies that we can extend $\s$ to 
$$K(a_i^\xi:(\xi,i)\ineq (\tau,l))\to K(b_i^\xi:(\xi,i)\ineq (\tau,l))$$
so that $a^{\tau}_l\mapsto b^\tau_l$.  This yields the desired $dd$-kernel structure.

\bigskip

We have thus shown that the differential field $(E,\d)$ has a $dd$-kernel structure. In other words, there is differential-homomorphism 
$$\s:K(a_i^\xi: (\xi,i)\in \Nn)\to K(b_i^\xi: (\xi,i)\in \Nn)$$ 
extending that on $K$. This the desired $dd$-kernel extension of length $(\infty,1)$.
\end{proof}

At this point one might ask: can the $dd$-kernel $L_{\infty,s}$ obtained in the theorem be extended to $L_{\infty,\infty}$? In other words, do the hypotheses of the theorem imply the existence of a $dd$-realisation of $L_{(r,s)}$? It turns out this is not generally possible. In Example~\ref{mainex} we saw that such a $dd$-kernel extension need not exist in positive characteristic. But we also pointed out there (paragraph before the example) that in characteristic zero this extension can be found. We make this precise in the following corollary.

\begin{corollary}\label{corddreal}
Assume $char(K)=0$. Let $n,s \geq 1$, $M\in \NN$, and $r\geq(ns+1)(M+1)$. Suppose $L_{(r,s)}$ is a $dd$-kernel over $K$ in $n$-variables such that all minimal separable leaders and minimal separable $a$-leaders are in $L_{(M,s)}$. Then, $L_{(r,s)}$ has a $dd$-realisation.
\end{corollary}
\begin{proof}
As we are in characteristic zero, there are no ($a$-)inseparable leaders, and the above theorem yields a $dd$-kernel extension $L_{\infty,s}$ which note is a differential field and so are $L_{(\infty,s)}^a$ and $L_{(\infty,s)}^b$. Moreover, the homomorphism of $L_{(\infty,s)}$ is a differential-homomorphism from $L_{(\infty,s)}^a\to L_{(\infty,s)}^b$. Now let $(F,\d)$ a sufficiently saturated model of DCF$_0$ extending $(L_{(\infty,s)},\d)$, then by quantifier elimination $\s:L_{(\infty,s)}^a\to L_{(\infty,s)}^b$ is a partial elementary map of $(F,\d)$ and hence can be extended to a differential automorphism of $F$, yielding the desired $dd$-realisation.
\end{proof}

We can now formulate an alternative axiomatisation for DCF$_0$A (cf. the original axiomatisation presented in \cite{Bu}).

\begin{proposition}\label{thedcomp}
Let $(K,\d,\s)$ be a differential-difference field of $char(K)=0$. Then, $(K,\d,\s)$ is existentially closed if and only if the following condition is satisfied:
\begin{enumerate}
\item [($\dagger$)] for each $n,s \geq 1$, $M\in \NN$, and $r\geq(ns+1)(M+1)$, if $L_{(r,s)}$ is a $dd$-kernel over $K$ in $n$-variables such that all minimal separable leaders and minimal separable $a$-leaders are in $L_{(M,s)}$, then $(a_i^{\xi,u}:(\xi,u,i)\in \Gamma(r,s))$ has an algebraic specialisation over $K$ of the form $(\d^\xi\s^u c_i:(\xi,i)\in \Gamma(r))$ for some $n$-tuple $(c_1,\dots,c_n)$ from $K$. 
\end{enumerate}
\end{proposition}
\begin{proof}
($\Rightarrow$) Suppose that $L_{(r,s)}$ is a $dd$-kernel over $(K,\d,\s)$ in $n$-variables satisfying condition ($\dagger$). Then, Corollary~\ref{corddreal} yields a $dd$-realisation of $L_{(r,s)}$, call it $(F,\d,\s)$. Since $(K,\d,\s)$ is existentially closed in $(F,\d,\s)$, the Zariski-locus of 
$$(a_i^{\xi,u}:(\xi,u,i)\in \Gamma(r,s))=(\d^\xi \s^u a_i:(\xi,u,i)\in \Gamma(r,s))$$
over $K$ has a point of the form $(\d^\xi \s^u c_i:(\xi,u,i)\in \Gamma(r,s))$ for some $n$-tuple $(c_1,\dots,c_n)$ from $K$. This yields the desired specialisation.

\

($\Leftarrow$) Let $\phi(\bar x)$ be a quantifier-free $\Ldd(K)$-formula with a realisation $\bar a$ in some extension $(F,\d,\s)$. By possibly adding more variables, we may assume that $\phi(\bar x)$ is of the form 
$$\psi(\d^\xi x_i,\d^\xi\s x_i: \xi\leq m, i\leq n)$$ 
for some $n,m\geq 1$ and $\psi$ a system of polynomial equations over $K$.

Consider the differential field
$$L=K(\d^\xi a_i, \d^\xi \s a_i: (\xi,i)\in \Nn).$$
By Theorem~\ref{finiteness}, the set of minimal-separable leaders and the set of minimal-separable $a$-leaders of $L$ are finite (and as we are in characteristic zero there are no inseparable leaders). Thus, we may choose $M\geq m$ large enough such that: all minimal-separable leaders and minimal-separable $a$-leaders of $L$ are in $L_M$ (where we view $L_M$ as a differential kernel in $2n$ variables). 

Now let $r\geq (n+1)(M+1)$ and note that then $r\geq m$. Then, $L_r=L_{(r,1)}$ is a $dd$-kernel of length $(r,1)$ satisfying the hypotheses of Corollary~\ref{corddreal} with $s=1$, and thus there is an algebraic specialisation of $(\d^\xi a_i,\d^\xi\s a_i:(\xi,i)\in \Gamma(r))$ of the form $(\d^\xi c_i,\d^\xi \s c_i:(\xi,i)\in \Gamma(r))$ for some $n$-tuple $\bar c=(c_1,\dots,c_n)$ from $K$. Then, the tuple
$$(\d^\xi c_i,\d^\xi \s c_i:(\xi,i)\in \Gamma(s))$$
is a solution of $\psi$ and consequently $\bar c$ is realisation of $\phi$ in $K$.
\end{proof}

\begin{remark}
From the proof of the right-to-left implication, we observe that (by the standard linearisation process) condition ($\dagger$) of the theorem can be restricted to $dd$-kernels of length $(r,1)$; i.e., suffices to consider the case $s=1$.  
\end{remark}

In the next section we explain how to resolve the issues/difficulties that we encounter in characteristic $p>0$.

\section{On $dd$-realisations in characteristic $p>0$}

We now assume our fields are of characteristic $p>0$ and we fix $(K,\d,\s)$ a differential-difference field. In the previous section we noted that in positive characteristic the hypotheses of Theorem~\ref{ddreal} do not imply the existence of a $dd$-realisation, see Example~\ref{mainex}. In this section we expand on these hypotheses to provide conditions that guarantee the existence of $dd$-realisations. The key ingredients here are Proposition~\ref{fordifference} and the following result of Chatzidakis \cite{Chatzi}.

\begin{lemma}\label{chat2}\cite[Lemma 2.5]{Chatzi}
Let $(k,\s)$ be a difference field. Let $U$ and $V$ be irreducible (affine) algebraic varieties over $k$ of dimension $d$ such that $V\subseteq U\times \sigma(U)$ and $V$ projects generically onto $U$ and $\sigma(U)$. Write $k(V)=k(x_1,\dots,x_n,y_1,\dots,y_n)$ and assume that there is $t\leq d$ such that $(x_1,\dots,x_t,y_{t+1},\dots,y_d)$ is a separating transcendence basis of $k(V)$ over $k$ and $x_{t+1},\dots,x_d$ are separably algebraic over $k(x_1^p,\dots,x_t^p,y_{t+1},\dots,y_d)$. Then, there is $\bar a\in U$ such that $k(\bar a)_\sigma$ is separable extension of $k$ and $(\bar a,\sigma(\bar a))$ is generic in $V$ over $k$.
\end{lemma}

We can now prove one of our main results. It provides sufficient conditions that guarantee the existence of a $dd$-realisation of a $dd$-kernel $L_{(r,s)}$ (in characteristic $p>0$). The reader might wish to compare this with Theorem~\ref{ddreal} where only a $dd$-kernel extension of length $(\infty,s)$ was obtained.

\begin{theorem}\label{ddrealp}
Let $n\geq 1$, $M\in \NN$, $s\geq (M+1)\, n$, and $r\geq (M+1)(ns+1)$. Suppose $L_{(r,s)}$ is a $dd$-kernel over $(K,\d,\s)$ in $n$-variables such that all minimal separable leaders, inseparable leaders, minimal separable $a$-leaders, and inseparable $a$-leaders are in $L_{(M,s)}$. Furthermore, let $0\leq t\leq d\in\NN$, $I\subseteq \{0,1,\dots,M\}\times \nn$ and $I^c$ the complement of $I$ in $\{0,1,\dots,M\}\times \nn$, and suppose
\begin{enumerate}
\item $B:=\{a_i^{\xi,u}:\, (\xi,i)\in I \text{ and } u\leq s\}$ is algebraically independent over $K$, and for each $(\xi,i)\in I^c$, the element $a_i^{\xi,1}$ is algebraic over the field
$$K(B)(a_i^{\xi,0}:(\xi,i)\in I^c).$$
\item $d=trdeg_{K(B)}K(B)(a_i^{\xi,0}:(\xi,i)\in I^c)$ and there is an enumeration $x_1,\dots,x_m$ of the set $\{a_i^{\xi,u}:(\xi,i)\in I^c, 0\leq u\leq s-1\}$ such that, writing $y_j=a_i^{\xi,u+1}$ when $x_j=a_i^{\xi,u}$ for $j=1,\dots,m$, $(x_1,\dots,x_t,y_{t+1},\dots,y_d)$ is a separating transcendence basis of $K(B)(x_1,\dots,x_m,y_1,\dots,y_m)$ over $K(B)$ and $x_{t+1},\dots,x_d$ are separably algebraic over $K(B)(x_1^p,\dots,x_t^p,y_{t+1},\dots,y_d)$.
\end{enumerate}
Then, $L_{(r,s)}$ has a $dd$-realisation.
\end{theorem}

\begin{proof}
Introduce new generic (algebraically independent) elements $\{a_i^{\xi,u}:(\xi,i)\in I,\, u>s\}$; extend $\s$ from $K$ to a difference field $k:=K(a_i^{\xi,u}:(\xi,i)\in I,\, u\geq 0)$ with $\s(a_i^{\xi,u})=a_i^{\xi,u+1}$. Note that properties (1) and (2) that hold over $K(B)$ continue to hold over $k$. We may apply Lemma~\ref{chat2} (with $U$ and $V$ the locus of $(x_1,\dots,x_m)$ and $(x_1,\dots,x_m,y_1,\dots,y_m)$ over $k$, respectively), which yields a \emph{separable} extension of $k$ of the form
$$k(a_i^{\xi,u}:\, (\xi,u)\in I^c, \, u\geq 0)$$
with a difference structure extending that of $k$ such that $\s(a_i^{\xi,u})=a_i^{\xi,u+1}$.

On the other hand, by Theorem~\ref{ddreal} and the fact that $r\geq (M+1)(ns+1)$, we can extend $L_{(r,s)}$ to a $dd$-kernel $L_{\infty,s}$. We now build a $dd$-kernel extension of length $(\infty,s+1)$. We first claim that there is an extension of the derivation on $L_{(\infty,s)}$ to
$$L_{\infty,s}(a_i^{\xi,s+1}:0\leq \xi< M, \, 1\leq i\leq n)\to L_{\infty,s}(a_i^{\xi,s+1}:0\leq \xi\leq M, \, 1\leq i\leq n).$$
such that $\delta(a_i^{\xi,s+1})=a_i^{\xi+1,s+1}$. By Proposition~\ref{fordifference}, which applies as $k(a_i^{\xi,u}:\, (\xi,u)\in I^c, \, u\geq 0)$ is separable over $k$ and $s\geq (M+1)\, n$, for each $a_i^{\xi,s+1}$ with $\xi<M$ there are precisely three possibilities
\begin{enumerate}
\item [(i)] $(\xi,i)\in I$; in which case $a_i^{\xi,s+1}$ is transcendental over $L_{(\infty,s)}(a_i^{\tau,s+1}:i\leq n, \tau<\xi)$,
\item [(ii)] $a_i^{\xi,s+1}$ is inseparable over $L_{(\infty,s)}(a_i^{\tau,s+1}:i\leq n, \tau<\xi)$ with minimal polynomial having constant coefficients,
\item [(iii)] $a_i^{\xi,s}$ is separably algebraic over $L_{(\infty,s-1)}(a_i^{\tau,s}:i\leq n, \tau<\xi)$; in which case $a_i^{\xi,s+1}$ is separably algebraic over $L_{(\infty,s)}(a_i^{\tau,s+1}:i\leq n, \tau<\xi)$.
\end{enumerate}
In cases (i) and (ii), by Lemma~\ref{keyforext}(ii)\&(iii), we can extend the derivation arbitrarily on $a_i^{\xi,s+1}$; so we choose it as $a_i^{\xi+1,s+1}$. In case (3), the same argument as in Case (1) (inside Claim 2) of the the proof of Theorem~\ref{ddreal}, yields that $a_i^{\xi+1,s+1}$ is the image of $a_i^{\xi,s+1}$ of the unique extension of the derivation. This yields the desired derivation from $L_{\infty,s}(a_i^{\xi,s+1}:0\leq \xi< M, \, 1\leq i\leq n)$ to $L_{\infty,s}(a_i^{\xi,s+1}:0\leq \xi\leq M, \, 1\leq i\leq n)$.

We now wish to define the element $a_i^{M+1,s+1}$ in a suitable fashion. By the assumption on minimal separable and inseparable leaders of $L_{(r,s)}$, we only have two cases for $a_i^{M+1,s}$; that is:
\begin{enumerate}
\item [(a)] $a_i^{M+1,s}$ is not a leader of $L_{(M+1,s)}$, then set $a_i^{M+1,s+1}$ to be a new generic.
\item [(b)] $a_i^{M+1,s}$ is a separable leader of $L_{(M+1,s)}$; in which case $a_i^{M,s}$ is a separable leader and then set
$$a_i^{M+1,s+1}:=-\frac{(f^\d)^\s(a_i^{M,s+1})}{(f')^\s(a_i^{M,s+1})}$$
where $f$ is the minimal polynomial of $a_i^{M,s}$.
\end{enumerate}
In case (a), by our choice, $\sigma$ can be extended to $a_i^{M+1,s}\to a_i^{M+1,s+1}$. Furthermore, in this case, $a_i^{M,s}$ is either an inseparable leader or not a leader at all, and so, using Proposition~\ref{fordifference}, we obtain that $a_i^{M,s+1}$ is either an inseparable leader or not a leader at all; in either case, $\d$ can be extended to $a_i^{M,s+1}\to a_i^{M+1,s+1}$. On the other hand, if we are in case (b), the choice of $a_i^{M+1,s+1}$  is the unique one that yields an extension of $\d$ and $\s$ to $a_i^{M,s+1}$ and $a_i^{M+1,s}$, respectively, see again Case~(1) (inside Claim 2) of the the proof of Theorem~\ref{ddreal}.

In a similar fashion we may choose a suitable $a_i^{M+1,s+2}$ and so on; this yields the $dd$-kernel $L_{(\infty,s+1)}$ extending $L_{\infty,s}$. We may repeat this procedure to obtain a $dd$-kernel extension $L_{(\infty,s+2)}$, and iterating this we construct the desired $dd$-realisation $L_{\infty,\infty}$ of $L_{(r,s)}$.
\end{proof}

\

\section{The model-companion DCF$_p$A}

We are now in the position to prove the main result of the paper (Theorem~\ref{thecomp}). Namely, that the theory of differential-difference fields in characteristic $p>0$ has a model-companion. 
In the course of our proof we will make use of the following result of Chatzidakis \cite{Chatzi}.

\begin{lemma}\label{chat1}\cite[Lemma 2.1 and 2.4]{Chatzi} Let $(k,\sigma)$ be a difference field and $k(\bar a)_\sigma$ a separable extension of $k$.
\begin{enumerate}
\item There is a $\s$-transcendence basis $\bar b\subset \bar a$ of $k(\bar a)_\sigma$ over $k$ such that $k(\bar a)_\s$ is a separable extension of $k(\bar b)_\sigma$.
\item Assume $\sigma(\bar a)\in k(\bar a)^{\alg}$ and let $m\geq \text{trdeg}_k k(\bar a)$. Then, there is separating transcendence basis $\bar b\, \cup\,  \bar c$ of $k(\bar a,\sigma(\bar a),\dots,\sigma^m(\bar a))$ over $k$ such that 
$$\bar b\cup \sigma^{-1}(\bar c)\subseteq \bar a\cup \sigma(\bar a)\cup \cdots \cup \sigma^{m-1}(\bar a)$$
and $\sigma^{-1}(\bar c)$ is separably algebraic over $k(\bar b^p,\bar c)$.
\end{enumerate}
\end{lemma}

\

\begin{theorem}\label{thecomp}
Let $(K,\d,\s)$ be a differential-difference field of characteristic $p>0$. Then, $(K,\d,\s)$ is existentially closed if and only if the following three conditions are satisfied:
\begin{enumerate}
\item $(K,\d)$ is differentially closed,
\item $\s$ is an automorphism of $K$, and
\item if $L_{(r,s)}=K(a_i^{\xi,u}: (\xi,u,i)\in \Gamma(r,s))$ is a $dd$-kernel over $K$ satisfying the hypotheses of Theorem~\ref{ddrealp}, then $(a_i^{\xi,u}:(\xi,u,i)\in \Gamma(r,s))$ has an algebraic specialisation over $K$ of the form $(\d^\xi \s^u c_i:(\xi,u,i)\in \Gamma(r,s))$ for some tuple $\bar c$ from $K$. 
\end{enumerate}
\end{theorem}
\begin{proof}
($\Rightarrow$) By Corollary~\ref{conseq1}, $(K,\d)$ is differentially closed, and, by Lemma~\ref{extendtoauto}, $\sigma$ must be an automorphism. Now suppose that $L_{(r,s)}$ is a $dd$-kernel over $K$ satisfying the hypotheses of Theorem~\ref{ddrealp}. Then, that theorem yields a $dd$-realisation of $L_{(r,s)}$, call it $(F,\d,\s)$. Since $(K,\d,\s)$ is existentially closed in $(F,\d,\s)$, the Zariski-locus of 
$$(a_i^{\xi,u} :(\xi,u,i)\in \Gamma(r,s))=(\d^\xi\s^u a_i:(\xi,u,i)\in \Gamma(r,s))$$
over $K$ has a point of the form $(\d^\xi\s^u c_i:(\xi,u,i)\in \Gamma(r,s))$ for some tuple $\bar c$ from $K$. This yields the desired specialisation.

\

($\Leftarrow$) Let $\phi(\bar x)$ be a quantifier-free $\Ldd(K)$-formula with a realisation $\bar b$ in some extension $(F,\d,\s)$. We set
$$E_{(r,s)}=K(\d^\xi \s^u b_i: \, (\xi,u,i)\in \Gamma(r,s))$$
and note that this is a $dd$-kernel over $K$ for any $(r,s)$. By possibly adding more variables (and the standard linearisation process), we may assume that there is a large enough $M$ such that $\phi(\bar x)$ is of the form 
$$\psi(\d^\xi \s^u x_i: (\xi,u,i)\in \Gamma(M,1))$$ 
for $\psi$ a system of polynomial equations over $K$ and, furthermore, the minimal separable leaders of $E_{\infty,\infty}$ are all in $E_{M,1}$. As $(K,\d)$ is differentially perfect, $E_{\infty,\infty}$ is a separable extension of $K$, and so, by Theorem~\ref{finiteness}, we may further assume (by possibly taking a larger $M$) that the inseparable leaders and inseparable $a$-leaders of $E_{\infty,1}$ are all in $E_{M,1}$. 

Since all minimal separable leaders are in $E_{M,1}$, we may proceed as in proof of Theorem~\ref{ddrealp} to construct a new $dd$-kernel $L_{\infty,\infty}=K(a_i^{\xi,u}: (\xi,u,i)\in \Gamma(\infty,\infty))$ (which is in fact a differential-difference extension) such that $a_i^{\xi,u}=\d^\xi\s^u b_i$ when $\xi\leq M$ or $u\leq 1$ with the additional property that all minimal separable leaders and inseparable leaders of $L_{(\infty,\infty)}$ are in $L_{(M,\infty)}$.

By Lemma~\ref{chat1}(1), there is $I\subseteq \{0,1,\dots,M\}\times \nn$ such that $\{a_i^{\xi,0}:(\xi,i)\in I\}$ is a $\s$-transcendence basis for $L_{(M,\infty)}$ over $K$ and $L_{(\infty,\infty)}$ is separable over $K(a_i^{\xi,0}:(\xi,i)\in I)_\s$. In particular, for $(\xi,i)\in I^c$ (recall that $I^c$ denotes the complement of $I$ in $\{0,1,\dots,M\}\times \nn$) each $a_i^{\xi,0}$ is $\s$-algebraic over $K(a_i^{\xi,0}:(\xi,i)\in I)_\s$. Thus, after possibly adding new variables, we may assume that $a_i^{\xi,1}$ is algebraic over $K(B)(a_i^{\xi,0}:(\xi,i)\in I^c)$ where $B=\{a_i^{\xi,u}: (\xi,i)\in I, u\leq s\}$ for some big enough $s$. 

Assume the $s$ above is at least $(M+1)n$ and let $d=trdeg_{K(B)}K(B)(a_i^{\xi,0}:(\xi,i)\in I^c)$. By Lemma~\ref{chat1}(2), there is $t\leq d$ and an enumeration $x_1,\dots,x_m$ of the set $\{a_i^{\xi,u}:(\xi,i)\in I^c, 0\leq u\leq s-1\}$ such that, writing $y_j=a_i^{\xi,u+1}$ when $x_j=a_i^{\xi,u}$ for $j=1,\dots,m$, $(x_1,\dots,x_t,y_{t+1},\dots,y_d)$ is a separating transcendence basis of $K(B)(x_1,\dots,x_m,y_1,\dots,y_m)$ over $K(B)$ and $x_{t+1},\dots,x_d$ are separably algebraic over $K(B)(x_1^p,\dots,x_t^p,y_{t+1},\dots,y_d)$. 

Putting all the above together yields that, for $r\geq (M+1)(ns+1)$, the $dd$-kernel $L_{(r,s)}$ satisfies the hypotheses of Theorem~\ref{ddrealp}, and thus $(a_i^{\xi,u}:(\xi,u,i)\in \Gamma(r,s))$ has an algebraic specialisation over $K$ of the form $(\d^\xi \s^u c_i:(\xi,u,i)\in \Gamma(r,s))$ for some tuple $\bar c$ from $K$. Then, the tuple
$$(\d^\xi \s^u c_i:(\xi,u,i)\in \Gamma(M,1))$$
is a solution of $\psi$ and consequently $\bar c$ is realisation of $\phi$ in $K$.
\end{proof}

\begin{corollary}\label{finalresult}
The theory of differential-difference fields of characteristic $p>0$ has a model-companion. In other words, the companion theory DCF$_p$A exists.
\end{corollary}
\begin{proof}
Condition (3) of Theorem~\ref{thecomp} can be expressed as a scheme of first-order axioms in the language $\Ldd$. This is (nowadays) a standard procedure using bounds coming from the theory of polynomial rings \cite{VS} and the fact that extensions of derivations and automorphisms are \emph{determined} by algebraic equations. There are several presentations of these, see for instance \cite[proof of Theorem 2.6]{Chatzi}, \cite[\S 2]{Kow} and \cite[\S 4]{Tr}. We recommend in particular \cite[proof of Corollary 4.6]{Pie}, and leave the details to the interested reader.

Let us point out that indeed the model companion for differential-difference fields does yield the existence of (and in facts agrees with) DCF$_p$A. This follows immediately from the fact that for any differential-difference field extension $(L,\d,\s)$ of $(K,\d,\s)$, there exists a extension $(F,\d,\s)$ of $(L,\d,\s)$ with $(F,\d)\models$DCF (see Lemma~\ref{satext}). Thus, the model companion also axiomatises the class of differential-difference fields that are existentially closed in differential-difference extensions where the underlying differential fields are differentially closed. 
\end{proof}


From \cite{Bu}, we know that the theory DCF$_0$A is supersimple. We finish by making some observations in this direction in the case of characteristic $p>0$. As we noted in Section~\ref{preli}, the theory DCF$_p$ is complete, model-complete, and stable. From well known results of Chatzidakis and Pillay \cite{CP}, we obtain the following.

\begin{corollary}  Let $p>0$ and $(K,\d,\s)\models$\, DCF$_p$A.
\begin{enumerate}
\item The completions of the theory DCF$_{p}$A are determined by the action of the automorphism in the algebraic closure of the prime field $\mathbb F_p$.
\item If $A\subset K$, then $\acl(A)$ is the separable closure of the differential perfect closure in $K$ of $\{\s^i a:a\in A, i\in \mathbb Z\}$.
\item Let $a$ and $b$ be $n$-tuples from $K$ and $A\subset K$. Then, $\operatorname{tp}(a/A)=\operatorname{tp}(b/A)$ if and only if there is a differential-difference isomorphism from $\acl(A,a)\to \acl(A,b)$ fixing $A$ and sending $a$ to $b$.
\item Each completion of DCF$_p$A is simple but not supersimple (as the underlying field is not perfect, see Corollary~\ref{conseq1}). 
\end{enumerate}
\end{corollary}

\


\end{document}